\documentclass[10pt, a4paper, oneside, reqno]{amsart}
\usepackage{amsthm}
\usepackage{amsfonts}
\usepackage{amsmath}
\usepackage{amssymb}
\usepackage{mathrsfs}
\usepackage{euscript}   
\usepackage{color}
\usepackage{graphicx}

\usepackage{float}
\usepackage{amscd}
\usepackage{amsxtra}
\usepackage{hyperref}

\usepackage{tikz}
\usepackage{url}
\usepackage{enumerate}
\usepackage{mathrsfs}
\usepackage{epsfig}
\usepackage[active]{srcltx}
\usetikzlibrary{arrows}
\usetikzlibrary{calc}
\theoremstyle{plain}
\newtheorem{theorem}{Theorem}[section]

\newtheorem{lemma}[theorem]{Lemma}
\newtheorem{proposition}[theorem]{Proposition}
\newtheorem{corollary}[theorem]{Corollary}

\theoremstyle{definition}

\newtheorem{example}[theorem]{Example}
\newtheorem{question}[theorem]{Question}

\theoremstyle{remark}
\newtheorem{remark}[theorem]{Remark}
\newtheorem*{propositionA*}{$\mathrm{Proposition\,A}$}
\newtheorem*{remarkA*}{$\mathrm{Remark\,A}$}
\newtheorem*{propositionB*}{$\mathrm{Proposition\,B}$}
\newtheorem*{exampleA*}{$\mathrm{Example\,A}$}
\newtheorem*{theoremA*}{$\mathrm{Theorem\,A}$}
\newtheorem*{lemmaA*}{$\mathrm{Lemma\,A}$}
\newtheorem*{corollaryA*}{$\mathrm{Corollary\,A}$}
\newcommand*{\D}{\mathrm{d\hspace{.1ex}}}
\newcommand{\ncom}{\newcommand}
\ncom{\inp}[2]{\langle{#1},\,{#2} \rangle}

\newcommand*{\Ge}{\geqslant}
\newcommand*{\Le}{\leqslant}
\newcommand*{\supp}{\mathrm{supp}}
\makeatletter
\@namedef{subjclassname@2020}{%
\textup{2020} Mathematics Subject Classification}
\makeatother

\ncom{\bq}{\begin{equation}}
\ncom{\eq}{\end{equation}}
\ncom{\beqn}{\begin{eqnarray*}}
\ncom{\eeqn}{\end{eqnarray*}}
\ncom{\beq}{\begin{eqnarray}}
\ncom{\eeq}{\end{eqnarray}}
\ncom{\nno}{\nonumber}
\ncom{\rar}{\rightarrow}
\ncom{\Rar}{\Rightarrow}
\ncom{\noin}{\noindent}
\ncom{\bc}{\begin{centre}}
\ncom{\ec}{\end{centre}}
\ncom{\sz}{\scriptsize}
\ncom{\rf}{\ref}
\ncom{\sgm}{\sigma}
\ncom{\Sgm}{\Sigma}
\ncom{\lmd}{\lambda}
\ncom{\Lmd}{\Lambda}
\ncom{\eps}{\epsilon}
\ncom{\pcc}{\stackrel{P}{>}}
\ncom{\dist}{{\rm\,dist}}
\ncom{\im}{{\rm Im\,}}
\ncom{\sgn}{{\rm sgn\,}}
\ncom{\ba}{\begin{array}}
\ncom{\ea}{\end{array}}
\ncom{\eof}{\hfill{{\rule{1.5mm}{1.5mm}}}}
\ncom{\hone}{\mbox{\hspace{1em}}}
\ncom{\htwo}{\mbox{\hspace{2em}}}
\ncom{\hthree}{\mbox{\hspace{3em}}}
\ncom{\hfour}{\mbox{\hspace{4em}}}
\ncom{\hsev}{\mbox{\hspace{7em}}}
\ncom{\vone}{\vskip 2ex}
\ncom{\vtwo}{\vskip 4ex}
\ncom{\vonee}{\vskip 1.5ex}
\ncom{\vthree}{\vskip 6ex}
\ncom{\vfour}{\vspace*{8ex}}
\ncom{\integ}[4]{\int_{#1}^{#2}\,{#3}\,d{#4}}
\ncom{\vspan}[1]{{{\rm\,span}\#1 \}}}
\ncom{\dm}[1]{\displaystyle {#1}}

\begin{document}

\title[Subnormality of the quotients]{Subnormality of the quotients of \\ $\mathbb T^d$-invariant Hilbert modules}
\author{K. S. Amritha, S. Bera, S. Chavan, S. S. Sequeira}
\address{Department of Mathematics and Statistics,
	Indian Institute of Technology Kanpur, Kanpur, India.}
\email{\small amrithaks22@iitk.ac.in, santu20@iitk.ac.in, chavan@iitk.ac.in, shanolas@iitk.ac.in}

\thanks{The first author is supported by the PMRF Scheme $\mathsf{2302755}$, and the second author is supported by the FARE Scheme $\mathsf{2508001}$}


\begin{abstract}  
In this paper, we investigate $\mathbb T^d$-invariant Hilbert modules 
$\mathscr H$ over the polynomial ring $\mathbb C[z_1, \ldots, z_d]$ and their quotients, with primary emphasis on the classification of subnormal quotient modules of the form $\mathscr H/[p],$ where $p$ is a homogeneous polynomial in $d$ complex variables. The motivation for this classification arises from the case $p(z_1, z_2)=z_1-z_2,$ in which 
the subnormality of the quotient module $\widehat{\mathscr H_{\kappa_1} \otimes \mathscr H_{\kappa_2}}/[p]$ is equivalent to that of the module tensor product $\mathscr H_{\kappa_1} \otimes_{\mathbb C[z]} \mathscr H_{\kappa_2}$ of 
$\mathbb T$-invariant Hilbert modules $\mathscr H_{\kappa_1}$ and $\mathscr H_{\kappa_2}$, a problem first considered by N. Salinas. 
In addition to general structural results on principal homogeneous submodules $[p]$ of $\mathscr H$, we prove that if $\mathscr H/[p]$ is subnormal, then $p$ must be square-free.
Furthermore, when 
$\mathscr H$ is either $H^2(\mathbb D^d)$ or $H^2(\mathbb B^d),$ $d \Ge 1,$ 
the subnormality of the quotient module $\mathscr H/[p]$ implies that $\deg \,p \Le 1.$ We further show that $H^2(\mathbb D^2)/[p]$ (resp. $H^2(\mathbb B^2)/[p]$) is subnormal if and only if $\deg \,p \Le 1.$ If $H^2_d$ denotes the Drury–Arveson module in $d$ dimensions, then $H^2_2/[p]$ is subnormal if and only if $p$ is nonzero and $\deg \,p \Le 1$. This is surprising, especially since $H^2_d$ is not a subnormal Hilbert module for $d \Ge 2.$ 
Moreover, the phenomenon above does not occur for the Dirichlet module $D_2(\mathbb B^2)$.
Finally, we present an example demonstrating that a $\mathcal U_d$-invariant subnormal Hilbert module $\mathscr H$ may have a subnormal quotient module $\mathscr H/[p]$ even when $\deg p = 2.$
\end{abstract}

\maketitle
\tableofcontents


\section{Rotation-invariant Hilbert modules} \label{S1}
In the 1980s, R. G. Douglas initiated a program that employed module-theoretic techniques in multivariable operator theory, advocating the use of module language and tools from commutative algebra and algebraic geometry to develop a model theory for tuples of commuting operators (see \cite{D1984, D1986}). For an overview of these topics and more recent advances, the reader is referred to \cite{CG2003, DP1989, Sa2015, Sb2015}. The investigations in this paper follow a similar line of inquiry, with a focus on subnormal quotient modules.

The set of nonnegative integers is denoted by $\mathbb Z_+$. Let $\mathbb D$ and $\mathbb T$ denote the open unit disc and the unit circle in the complex plane $\mathbb C,$ respectively. Fix a positive integer $d.$ For $z=(z_1, \ldots, z_d) \in \mathbb C^d$ and $\alpha=(\alpha_1, \ldots, \alpha_d) \in \mathbb Z^d_+,$ let $z^\alpha := z^{\alpha_1}_1\cdots z^{\alpha_d}_d,$ $\alpha! :=\alpha_1!\cdots\alpha_d!$ and $|\alpha| :=\alpha_1 + \cdots + \alpha_d$. For $\alpha=(\alpha_1, \ldots, \alpha_d)$ and $\beta =(\beta_1, \ldots, \beta_d) \in \mathbb Z^d_+,$ we say that $\alpha \Le \beta$ if $\alpha_j \Le \beta_j$ for every $j=1, \ldots, d.$ The symbol $\inp{z}{w}$ denotes the standard inner product on $\mathbb C^d$, and $\mathbb C[z_1,\ldots, z_d]$ denotes the ring of polynomials in $z_1,\ldots, z_d$ with complex coefficients.  
We say that $p \in \mathbb C[z_1, \ldots, z_d]$ is a {\it homogeneous polynomial of degree} $m \in \mathbb Z_+$ (denoted by $\deg\,p=m$) if $p(\lambda z_1, \ldots, \lambda z_d)=\lambda^m p(z_1, \ldots, z_d)$ for every $\lambda \in \mathbb C.$ A polynomial $p \in \mathbb C[z_1, \ldots, z_d]$ is said to be {\it square-free} if there is no nonconstant polynomial $q \in \mathbb C[z_1, \ldots, z_d]$ such that $q^2$ divides $p.$ The closed linear span of a nonempty subset $E$ of a normed linear space is denoted by $\bigvee E$.
For a nonempty set $X$ and $x \in X,$ let $\delta_x$ denote the Borel probability measure concentrated at $x.$ 
Let $\ell$ denote the Lebesgue measure on $[0, 1]$. 
Given measurable spaces $X$ and $Y,$ a measure $\mu$ on $X,$ and
a measurable function $\varphi : X \rar Y,$  the push-forward of $\mu$ by $\varphi$ is denoted by $\varphi_*\mu.$ The support of a positive Borel measure $\mu$ is denoted by $\supp \, \mu.$ Let $\Omega$ be a {\it domain} in $\mathbb C^d$, that is, a nonempty open connected subset of $\mathbb C^d.$ For a function $f : \Omega \rar \mathbb C$, let $Z(f)$ denote the zero set of $f.$
We say that $\Omega$ is a {\it Reinhardt domain} if it is invariant under the natural action of the $d$-torus $\mathbb T^d$: For every $z=(z_1, \ldots, z_d) \in \Omega$ and every $\zeta=(\zeta_1, \ldots, \zeta_d) \in \mathbb T^d$, the point $\zeta \cdot z := (\zeta_1 z_1, \ldots, \zeta_d z_d)$ belongs to $\Omega$.
For a Reinhardt domain $\Omega$ in $\mathbb C^d$ and a function $f : \Omega \rar \mathbb C$, define $\overline{f} : \Omega \rar \mathbb C$ by
$\overline{f}(z)=\overline{f(\overline{z})},$ where $\overline{a}:=(\overline{a}_1, \ldots, \overline{a}_d)$ for $a=(a_1, \ldots, a_d) \in \mathbb C^d,$ and $\overline{b}$ denotes the complex conjugate of $b \in \mathbb C.$

For a complex Hilbert space $\mathcal H,$ let $\mathcal B(\mathcal H)$ denote the unital $C^*$-algebra of bounded linear operators on $\mathcal H,$ where the identity operator $I$ is the unit, composition of operators is
the multiplication, and the uniquely defined adjoint $T^*$ of a bounded linear operator $T$ on $\mathcal H$ is the involution.  
A {\it commuting $d$-tuple $T$ on $\mathcal H$} is a $d$-tuple 
$(T_1, \ldots, T_d)$ of operators $T_1, \ldots, T_d \in \mathcal B(\mathcal H)$ satisfying $T_iT_j=T_jT_i$ for $1 \Le i \ne j \Le d$. The Taylor spectrum of $T$ is denoted by $\sigma(T).$ For the definition and basic
theory of the Taylor spectrum, the reader
is referred to \cite{C1988, T1970}. The polynomial functional calculus of $T$ makes the Hilbert space $\mathcal H$ into a {\it Hilbert module over $\mathbb C[z_1, \ldots, z_d]$}:
\beq \label{scalar-multi}
{\mathfrak m}_T(p, h) := p(T)h, \quad p \in \mathbb C[z_1, \ldots, z_d], ~h \in \mathcal H.
\eeq
The Hilbert module $\mathcal H$ with the {\it scalar multiplication} ${\mathfrak m}_T$ is denoted by the pair 
$(\mathcal H, {\mathfrak m}_T).$ We say that $(\mathcal H, {\mathfrak m}_T)$ is {\it subnormal} if ${\mathfrak m}_T(p, \cdot)=p(T)$ is subnormal for every $p \in \mathbb C[z_1, \ldots, z_d].$ By a result of E. Franks (see \cite[Theorem~0.2]{F1994}), $(\mathcal H, {\mathfrak m}_T)$ is a subnormal Hilbert module if and only if $T$ is subnormal. Recall that the commuting $d$-tuple $T=(T_1, \ldots, T_d)$ on $\mathcal H$ is {\it subnormal} if
there exists a Hilbert space $\mathcal K$ and a commuting $d$-tuple 
$N=(N_1, \ldots, N_d)$ on $\mathcal K$ consisting of normal
operators such that $\mathcal H \subseteq \mathcal K$ (as an isometric
embedding) and $T_j = {N_j}|_{\mathcal H}$ for $j=1, \ldots, d$; such a commuting $d$-tuple $N$ is called a {\it normal extension} of $T$. A normal extension $N$ of $T$ is said to be {\it minimal} if $\mathcal K$ is the only
subspace that contains $\mathcal H$ and reduces each $N_j$ for $j=1, \ldots, d.$

Let $\Omega$ be a bounded, Reinhardt domain in $\mathbb C^d$ containing the origin.
By a {\it $\mathbb T^d$-invariant Hilbert module $\mathscr H_\kappa$ on $\Omega$}, we mean a reproducing kernel Hilbert space of complex-valued holomorphic functions on $\Omega$ with a reproducing kernel $\kappa : \Omega \times \Omega \rightarrow \mathbb C$, which satisfies the following conditions: 
\begin{enumerate}
\item[({\tiny C1})] $\kappa$ does not vanish at the origin,
\item[({\tiny C2})] $\kappa(\zeta \cdot z, \,\zeta \cdot w)=\kappa(z, w)$
for all $\zeta \in \mathbb T^d$ and $z, w \in \Omega,$
\item[({\tiny C3})] $\mathscr H_\kappa$ is a Hilbert module with scalar multiplication ${\mathfrak m}$ given by $$
{\mathfrak m}(p, f) := p \cdot f, \quad p \in \mathbb C[z_1, \ldots, z_d], ~f \in \mathscr H_\kappa,$$ 
\item[({\tiny C4})] polynomials in $z_1, \ldots, z_d$ are dense in $\mathscr H_\kappa.$
\end{enumerate}
Let $\mathscr H_\kappa$ be a $\mathbb T^d$-invariant Hilbert module on $\Omega$; when the domain $\Omega$ plays no role, we simply call it a $\mathbb T^d$-invariant Hilbert module. By \cite[Theorem~2.11]{CY2015}, together with assumptions ({\tiny C1}) and ({\tiny C2}), there exists a multi-sequence $\{a_\alpha\}_{\alpha \in \mathbb Z^d_+}$ of positive real numbers such that 
\beqn
\kappa(z, w) = \sum_{\alpha \in \mathbb Z^d_+}a_\alpha \,z^\alpha \overline{w}^\alpha, \quad z, w \in \Omega,
\eeqn
and $\{z^\alpha\}_{\alpha \in \mathbb Z^d_+}$ forms an orthogonal basis for $\mathscr H_\kappa$.
Papadakis’s theorem (see \cite[Theorem~2.10]{PR2016}), combined with \cite[Proposition~4.11]{PR2016}, now shows that 
\beq \label{onbasis}
\mbox{$\{\sqrt{a_\alpha}\,z^\alpha\}_{\alpha \in \mathbb Z^d_+}$ forms an orthonormal basis for $\mathscr H_\kappa$;} 
\eeq
for a one-variable analog of this fact, see \cite[Theorem~4.12]{PR2016}. Moreover, 
\beq \label{kappa-cts}
\mbox{$\kappa : \Omega \times \Omega \rightarrow \mathbb C$ is jointly continuous.} 
\eeq
Indeed, by Hartogs' theorem (see \cite[Theorem*\,1.7.13]{JP2008}), the mapping
$(z, w) \mapsto \kappa(z, \overline{w})$ is
holomorphic on $\Omega \times \Omega$. 
Furthermore, by condition ({\tiny C3}), the coordinate functions $z_1, \ldots, z_d$ are {\it multipliers} of $\mathscr H_\kappa,$ that is, 
\beqn
f \in \mathscr H_\kappa ~\Rightarrow ~ z_jf \in \mathscr H_\kappa, \quad j=1, \ldots, d.
\eeqn
Note that the operator $\mathscr M_{z_j}$ of multiplication by the $j$th coordinate function $z_j,$ $j=1, \ldots, d,$ is a bounded linear operator on $\mathscr H_\kappa.$ Let $\mathscr M_z$ denote the commuting $d$-tuple $(\mathscr M_{z_1}, \ldots, \mathscr M_{z_d})$. Then
${\mathfrak m}={\mathfrak m}_{\mathscr M_z}$ (see \eqref{scalar-multi}).  

Let $\mathbb B^d$ denote the open unit ball in $\mathbb C^d$, and $\mathcal U_d$ denote the group of complex $d \times d$ unitary matrices. Note that $\mathcal U_d$ acts naturally on the unit ball $\mathbb B^d$: For every $z \in \mathbb B^d$ and every $U \in \mathcal U_d,$ the point $U\!\cdot \!z$ belongs to $\mathbb B^d$.
A $\mathbb T^d$-invariant Hilbert module $\mathscr H_\kappa$ on $\mathbb B^d$
is said to be {\it $\mathcal U_d$-invariant} if \beqn \kappa(U\!\cdot \!z,
\, U\!\cdot \!w)=\kappa(z, w), \quad z, w \in \mathbb B^d, ~U \in \mathcal U_d. \eeqn
See \cite[Theorem 2.12]{CY2015} for a characterization of $\mathcal U_d$-invariant Hilbert modules.

Let $\Omega$ be a domain in $\mathbb C$ and $\kappa_j : \Omega \times \Omega \rar \mathbb C,$ $j=1, 2,$ be two positive semi-definite kernels. Following \cite[Definition~5.12]{PR2016}, define $$\kappa_1 \otimes \kappa_2(z, w)=\kappa_1(z_1, w_1)\kappa_2(z_2, w_2), \quad z=(z_1, z_2), w=(w_1, w_2) \in \Omega \times \Omega.$$
The {\it tensor product} $\widehat{\mathscr H_{\kappa_1}\otimes\mathscr H_{\kappa_2}}$ of $\mathscr H_{\kappa_1}$ and $\mathscr H_{\kappa_2}$ is the Hilbert module $\mathscr H_{\kappa_1 \otimes \kappa_2}$ on $\Omega \times \Omega$ (see \cite[Theorem~5.11]{PR2016}). 
Following \cite[Definition~5.15]{PR2016}, define $$\kappa_1  \odot \kappa_2(z, w)=\kappa_1(z, w)\kappa_2(z, w), \quad z, w \in \Omega.$$
The {\it module tensor product} $\mathscr H_{\kappa_1} \otimes_{\mathbb C[z]} \mathscr H_{\kappa_2}$ of $\mathscr H_{\kappa_1}$ and $\mathscr H_{\kappa_2}$ is the Hilbert module $\mathscr H_{\kappa_1\odot \kappa_2}$ on $\Omega$ (see \cite[Corollary~3.6]{S1988}).

By a {\it submodule} $\mathcal M$ of $\mathscr H_\kappa,$ we mean a closed subspace $\mathcal M$ of $\mathscr H_\kappa$ that is invariant under ${\mathfrak m}_{\mathscr M_z}.$ Let $p$ be a homogeneous polynomial in $\mathbb C[z_1, \ldots, z_d].$
A {\it principal homogeneous submodule of $\mathscr H_\kappa$ generated by $p$} is given by $[p]:=\bigvee \{z^\alpha p : \alpha \in \mathbb Z^d_+\}$. In general, the inclusion $p \mathscr H_\kappa \subseteq [p]$ may be strict (see Example\,A). There is another submodule of $\mathscr H_\kappa$ naturally associated with a homogeneous polynomial $p$: 
$$[p]_0:=\{f \in \mathscr H_\kappa : f|_{Z(p)}=0\}.$$  Clearly,
$[p] \subseteq [p]_0$. 
In general, this inclusion may be strict (e.g. $[z^2_1] \subsetneq [z_1] =[z^2_1]_0$). However, 
if $\mathscr H_\kappa$ is a $\mathbb T^2$-invariant Hilbert module and $p$ is a nonzero homogeneous square-free polynomial, then $[p]_0 = [p].$ A proof of this fact, being peripheral to the main text, is relegated to the appendix (see Theorem\,A). 

Set $\mathscr H_\kappa /\mathcal M:=\mathscr H_\kappa \ominus \mathcal M$ and let
$P_\mathcal M$ (resp. $P^{\perp}_\mathcal M$) denote the orthogonal projection of 
$\mathscr H_\kappa$ onto $\mathcal M$ (resp. $\mathscr H_\kappa /\mathcal M$).
Consider the bounded linear operators $\mathscr T_{z_j, \mathcal M}$ on $\mathscr H_\kappa /\mathcal M$ given by 
\beqn
\mathscr T_{z_j, \mathcal M}:=P^{\perp}_{\mathcal M}\mathscr M_{z_j}|_{\mathscr H_\kappa /\mathcal M}, \quad j=1, \ldots, d.
\eeqn
It turns out that $\mathscr T_{z, \mathcal M}:=(\mathscr T_{z_1, \mathcal M}, \ldots,  \mathscr T_{z_d, \mathcal M})$ is a commuting $d$-tuple, and the scalar multiplication $\mathfrak m_{\mathscr T_{z}}$ satisfies
\beqn
{\mathfrak m}_{\mathscr T_{z}}(p, f) = P^{\perp}_\mathcal M p(\mathscr M_z)f, \quad p \in \mathbb C[z_1, \ldots, z_d], ~f \in \mathscr H_\kappa /\mathcal M
\eeqn 
(see Lemma~\ref{poly-calculus}(ii)). 
In particular, $\mathscr H_\kappa /\mathcal M$ is the Hilbert module $(\mathscr H_\kappa \ominus \mathcal M, {\mathfrak m}_{\mathscr T_z}).$ Since $\mathscr M_{z}\mathcal M \subseteq \mathcal M,$ the quotient $\mathscr H_\kappa/\mathcal M$ is invariant under $\mathscr M_{z}^*,$ and 
\beq
\label{adjoint}
\mathscr T^*_{z_j}=\mathscr M^*_{z_j}|_{\mathscr H_\kappa/\mathcal M}, \quad j=1, \ldots, d.  
\eeq
For simplicity, and provided it does not cause ambiguity, we drop the symbol $\mathcal M$ from the notations $P_\mathcal M,$ $P^{\perp}_\mathcal M$, $\mathscr T_{z, \mathcal M}$. 

Quotient modules on the polydisc or the unit ball have been previously investigated in \cite{CG2003, DGS2020, DM1993, DMV2000, DMV2001, FR2006, GW2007}, whereas subnormal Hilbert modules have been studied in \cite{AC2017, CY2015, CS1985, GHX2004, S1988}. The $\mathcal U_d$-invariant Hilbert modules have been treated in \cite{A1998, CY2015, GKMP2024, GHX2004}. For a brief discussion of essentially reductive quotients and submodules, see \cite[Section~5.5]{CM2020}.

In the operator-theoretic context, subnormality does not naturally propagate from a Hilbert module to its quotients; hence, there is no simple answer to the following question. 
\begin{question} \label{Q} 
For which principal homogeneous submodules $\mathcal M$ of a $\mathbb T^d$-invariant Hilbert module $\mathscr H_\kappa$ on $\Omega$ (respectively, a $\mathcal U_d$-invariant Hilbert module $\mathscr H_\kappa$ on $\mathbb B^d$), the quotient module $\mathscr H_\kappa /\mathcal M$ is subnormal?
\end{question}
\begin{remark} \label{rmk-one-v}
Let $p$ be a nonconstant homogeneous polynomial in $\mathbb C[z],$ that is, $p(z)=z^{k}$ for some integer $k \Ge 1.$ Let $\mathscr H_\kappa$ be a $\mathbb T$-invariant Hilbert module.
Note that $\big\{z^j: 0 \Le j \Le k-1\big\}$ is an orthogonal basis for $\mathscr H_\kappa /[p].$ Since 
$
\|\mathscr T_z(z^{k-1})\|=0 < \|\mathscr T_z^*(z^{k-1})\|$ whenever $k \Ge 2,$
it follows from \cite[Proposition~II.4.2]{Co1991} that $\mathscr H_\kappa /[p]$ is subnormal if and only if $\deg\,p=1.$ 
\eof
\end{remark}
In view of Remark~\ref{rmk-one-v}, we focus on Question~\ref{Q} for $d \Ge 2.$ This question is closely related to the following one posed by N. Salinas (see \cite[Remark 3.7]{S1988}): {\it Is the module tensor product $\mathscr H_{\kappa_1} \otimes_{\mathbb C[z]} \mathscr H_{\kappa_2}$ a subnormal Hilbert module when $\mathscr H_{\kappa_1}$ and $\mathscr H_{\kappa_2}$ are subnormal Hilbert modules?} 
This problem was subsequently investigated in the setting of weighted Bergman spaces, with the aim of clarifying the precise relationship between an algebraic notion (the module tensor product) and a geometric one (subnormality). The reader is referred to \cite[Corollary 1.8]{AC2017} for a negative answer. To see the precise connection between Salinas's question and Question~\ref{Q},
note that if $p(z_1, z_2)=z_1 - z_2$ and $\mathscr H_{\kappa_1},$ $\mathscr H_{\kappa_2}$ are subnormal $\mathbb T$-invariant Hilbert modules on $\mathbb D$, then $\widehat{\mathscr H_{\kappa_1}\otimes\mathscr H_{\kappa_2}}/[p]$ is subnormal if and only if $\mathscr H_{\kappa_1} \otimes_{\mathbb C[z]} \mathscr H_{\kappa_2}$ is subnormal (see Lemma~\ref{lem-sub-linear}(iv)).

\section{Statements of the main results} 

The following answers Question~\ref{Q} for all principal homogeneous submodules of the Hardy module on $\mathbb D^2$ (see Example~\ref{ex-Hardy}).

\begin{theorem} \label{Hardy-q-sub}
Let $p$ be a nonconstant homogeneous polynomial in $\mathbb C[z_1, z_2].$ Then the following are equivalent$:$
\begin{enumerate}
    \item[$\mathrm{(i)}$] the quotient module $H^2(\mathbb D^2)/[p]$ is subnormal,
    \item[$\mathrm{(ii)}$] $\deg \, p=1$,
    \item[$\mathrm{(iii)}$] there exist $a, b \in \mathbb C,$ not both zero, such that $a\mathscr T_{z_1, [p]}=b\mathscr T_{z_2, [p]}.$
\end{enumerate}
\end{theorem}

In some particular cases, one can also explicitly describe the representing measures for the moments of $\mathscr{T}_z$ (see Remark~\ref{rmk-rep-measures}). Also, Theorem~\ref{Hardy-q-sub} fails for $\mathbb T^2$-invariant Hilbert modules on $\mathbb D^2$ (see Example~\ref{fail-linear-bidisc}). 

The second main result of this paper answers Question~\ref{Q} for all principal homogeneous submodules of the Hardy module on $\mathbb B^2$ and the Drury–Arveson module (see Example~\ref{ex-Hardy}).
\begin{theorem} \label{Hardy-ball-q-sub}
Let $p$ be a nonconstant homogeneous polynomial in $\mathbb C[z_1, z_2],$
and let $\mathscr H$ be either the Hardy module $H^2(\mathbb B^2)$ or the 
Drury–Arveson module $H^2_2$. Then the following are equivalent$:$
\begin{enumerate}
    \item[$\mathrm{(i)}$] the quotient module $\mathscr H/[p]$ is subnormal,
    \item[$\mathrm{(ii)}$] $\deg \, p=1$,
    \item[$\mathrm{(iii)}$] there exist $a, b \in \mathbb C,$ not both zero, such that $a\mathscr T_{z_1, [p]}=b\mathscr T_{z_2, [p]}.$
\end{enumerate}
\end{theorem}


The implications (i)$\Rightarrow$(ii) and (ii)$\Rightarrow$(i) of Theorem~\ref{Hardy-ball-q-sub} fail for $\mathcal U_2$-invariant Hilbert modules on $\mathbb B^2$ (see Examples~\ref{fail-linear-ball} and \ref{Diri-module}). Nevertheless, a key ingredient in the proof of Theorem~\ref{Hardy-ball-q-sub}, namely, that the multiplication $d$-tuple $\mathscr M_z$ is an $m$-isometry, ensures that $\deg\,p \Le m$ (see Proposition~\ref{spherical-m-iso-new}). 
An important property of the Drury–Arveson module $H^2_d$ is that if $f(z) = g(z_1)$ for some $g \in H^2(\mathbb D),$ then $f \in H^2_d$ and $\|f\| = \|g\|_{H^2(\mathbb D)}$ (see \cite[Proposition~2.4]{H2023}). This property (cf. \eqref{weak-m-iso}) plays a crucial role in the proof of Theorem~\ref{Hardy-ball-q-sub}.
The proofs of the implication (iii) $\Rightarrow$ (ii) in Theorems~\ref{Hardy-q-sub} and \ref{Hardy-ball-q-sub} rely on decompositions of homogeneous polynomials that are valid only in two complex variables (see Proposition\,A and Remark\,A).

We next state some necessary or sufficient conditions for the subnormality of quotient modules of $\mathbb T^d$-invariant Hilbert modules. 
In the case where $\mathcal M$ is a principal homogeneous submodule, we have the following general result (see \cite[Lemma~5.5]{DGS2020} for a similar phenomenon in the context of boundary representations):

\begin{proposition} \label{rotation-q-sub}
Let $p$ be a nonconstant homogeneous polynomial in $\mathbb C[z_1, \ldots, z_d]$, and let $\mathscr H_\kappa$ be a $\mathbb T^d$-invariant Hilbert module.
If the quotient module $\mathscr H_\kappa/[p]$ is subnormal, then 
$p$ is square-free.  
\end{proposition}

As shown by Theorems~\ref{Hardy-q-sub} and \ref{Hardy-ball-q-sub}, the converse of Proposition~\ref{rotation-q-sub} is not true.
The following result yields numerous examples of non-isomorphic subnormal quotient modules of a 
$\mathcal U_2$-invariant Hilbert module (cf. Lemma~\ref{poly-calculus}(iii)).   
\begin{proposition} \label{sub-linear-u-invariant}
Let $p$ be a homogeneous polynomial in $\mathbb C[z_1, z_2]$ of degree $1,$ and 
let $\mathscr H_\kappa$ be a $\mathcal U_2$-invariant Hilbert module. 
If $\mathscr H_\kappa$ is subnormal, then the quotient module $\mathscr H_\kappa /[p]$ is also subnormal.
\end{proposition}

Here is the outline of the paper.
In Section~\ref{S4}, we provide proofs of Theorems~\ref{Hardy-q-sub} and~\ref{Hardy-ball-q-sub}, along with Propositions~\ref{rotation-q-sub} and \ref{sub-linear-u-invariant}. The requisite preliminary results are developed in Section~\ref{S3}. These include several noteworthy properties of principal homogeneous submodules of a $\mathbb T^d$-invariant Hilbert module (see Lemmata~\ref{poly-calculus} and~\ref{poly-calculus-new}), the construction of an orthonormal basis for a three-parameter family of quotient modules (see Lemma~\ref{coro-generators}), and a criterion for subnormality formulated in terms of a concrete Stieltjes moment problem (see Lemma~\ref{lem-sub-linear}). 
Section~\ref{S5} presents a collection of examples illustrating the main results of the paper (see Examples~\ref{fail-linear-bidisc}-\ref{exam-m-tensor}). 
We conclude the paper with several unresolved problems. 
Finally, the appendix establishes two strictly $2$-dimensional facts concerning the homogeneous polynomials used in the main text (see Propositions\,A and B) and discusses some applications to the Hilbert modules (see Theorem\,A and Corollary\,A). 
Table~\ref{Table1} summarizes the main results of this paper. Here, $p$ denotes a nonzero homogeneous polynomial in $\mathbb C[z_1, \ldots, z_d]$, 
$\mathscr S_d$ denotes the collection of subnormal Hilbert modules on $\Omega$,
and $\mathscr T\!\mathscr S_d$ (resp., $\mathscr U\!\mathscr S_d$) represents the collection of subnormal $\mathbb T^d$-invariant (resp., subnormal $\mathcal U_d$-invariant) Hilbert modules. 

\begin{table}[htbp]
\centering
\small
\setlength{\tabcolsep}{4pt}
\renewcommand{\arraystretch}{1}
\begin{tabular}{|c|c|c|}
\hline
$\mathscr H_\kappa$ 
& $\deg \,p \Le 1 \Rightarrow \mathscr H_\kappa/[p] \in \mathscr S_d$ 
& $\mathscr H_\kappa/[p] \in \mathscr S_d \Rightarrow \deg \,p \Le 1$ \\
\hline
$H^2(\mathbb D^2), H^2(\mathbb B^2), H^2_2$ 
& $\checkmark$ & $\checkmark$ \\
\hline
$H^2(\mathbb D^d), H^2(\mathbb B^d), H^2_d,\ d \Ge 3$ 
& $?$ & $\checkmark$ \\
\hline
$\mathscr H_\kappa \in \mathscr T\!\mathscr S_2$
& $\times$ & $\times$ \\
\hline
$\mathscr H_\kappa \in \mathscr U\!\mathscr S_2$
& $\checkmark$ & $\times$ \\
\hline
\end{tabular}
\vskip.15cm
\caption{\label{Table1} Subnormality of the quotient modules}
\end{table}

\section{Preparatory results} \label{S3}

We list below several basic properties of the $d$-tuple $\mathscr T_z$ on $\mathscr H_\kappa/[p]$,
the first of which appears in \cite[Proof of Theorem~6.2]{GW2007} without a proof, while a vast generalization of the third part for the Hardy module $H^2(\mathbb D^d)$ is given in \cite[Corollary~9]{DY1990}.

\begin{lemma}\label{poly-calculus} 
    Let $p, q, r \in \mathbb C[z_1, \ldots, z_d]$ be polynomials with $p$ nonconstant, and let $\mathscr H_\kappa$ be a $\mathbb T^d$-invariant Hilbert module. Then the following statements are true$:$
\begin{enumerate}
\item[$(\mathrm{i})$] $p(\mathscr T_{z, [p]})=0,$
\item[$(\mathrm{ii})$] $q(\mathscr T_{z, [p]})(P^\perp f) = P^\perp (qf)$ for any $f \in \mathscr H_\kappa,$ 
\item[$(\mathrm{iii})$] if $\mathscr T_{z, [p]}$ and $\mathscr T_{z, [q]}$ are unitarily equivalent, then $[p]=[q],$
\item[$(\mathrm{iv})$] if $p, q, r$ are homogeneous polynomials with $\deg\, r < \deg\, p$, then $r$ belongs to $\mathscr H_\kappa/[p],$ and  
\beqn
q(\mathscr T_{z, [p]})r = \begin{cases} qr & \mbox{if}~\deg \,q < \deg \, p - \deg \, r, \\
0 & \mbox{if}~p=qr,
\end{cases}
\eeqn
\item[$(\mathrm{v})$] if $p, r$ are homogeneous polynomials such that 
$\deg\, p \Ge \deg\,r + 2,$ then 
\beqn
\mathscr T_{z_j, [p]}r=z_jr, \quad j =1, \ldots, d.
\eeqn
\end{enumerate}
\end{lemma}
\begin{proof} 
(i) By \eqref{adjoint}, for $f, g \in \mathscr H_\kappa /[p],$ 
    \begin{align*}
    \inp{p(\mathscr T_z)f}{g}=\inp{f}{\overline{p}(\mathscr T_z^*)g}
        =\inp{f}{\overline{p}(\mathscr M_z^*)g}
        =\inp{p(\mathscr M_z)f}{g}
        =\inp{pf}{g}
        =0.
    \end{align*}
    Since $f, g$ are arbitrary, we obtain (i).

(ii) 
Let $f \in \mathscr H_\kappa.$ Since $z_iPf \in [p],$ 
\beq
\label{subeqn}
\mathscr T_{z_i}(P^\perp f)=P^\perp (z_if-z_iPf)=P^\perp (z_if), \quad i=1, \ldots, d.  
\eeq
After replacing $f$ by $z_j f \in \mathscr H_\kappa,$ we obtain
\beqn
\mathscr T_{z_i}(P^\perp(z_jf))=P^\perp (z_iz_jf), \quad i, j=1, \ldots, d.
\eeqn
It now follows from \eqref{subeqn} that $\mathscr T_{z_i}\mathscr T_{z_j}(P^\perp f)=P^\perp (z_iz_jf)$ for $i, j=1, \ldots, d.$
Applying this repeatedly, we obtain $z^\alpha(\mathscr T_z)(P^\perp f) = P^\perp (z^\alpha f)$ for any $\alpha \in \mathbb Z^d_+.$ Hence, by the linearity of $P^{\perp},$ we deduce (ii). 

(iii) 
Suppose $\mathscr T_{z, [p]}$ is unitarily equivalent to $\mathscr T_{z, [q]}.$ Then there exists a unitary operator $\mathscr U : \mathscr H_\kappa/[p] \to \mathscr H_\kappa/[q]$ such that $\mathscr U \mathscr T_{z_j, [p]} =\mathscr T_{z_j, [q]}\mathscr U,$ $j=1, \ldots, d.$ It now follows from part (i) that 
   $p(\mathscr T_{z, [q]}) =\mathscr U p(\mathscr T_{z, [p]})\mathscr U^{-1}=0,$ and hence by part (ii), 
\beqn
 P^\perp_{[q]}(p) = p(\mathscr T_{z, [q]})P^\perp_{[q]}(1)=0. 
\eeqn
This implies that $p \in [q]$, and hence $[p] \subseteq [q]$. Similarly, we obtain the inclusion $[q] \subseteq [p]$, which yields the equality $[p] = [q]$.

(iv) Let $s \in \mathbb C[z_1, \ldots, z_d]$ be a homogeneous polynomial such that $\deg\, s < \deg\, p.$ For $\alpha \in \mathbb Z^d_+,$ the degree of any monomial in $s$ (equal to $\deg\, s$) is less than the degree of any monomial in $pz^\alpha$ (equal to $\deg\,p+|\alpha|$). It follows that $\inp{s}{pz^{\alpha}}=0$ for all $\alpha \in \mathbb Z^d_+,$ and as the linear span of $\{pz^{\alpha}: \alpha \in \mathbb Z^d_+\}$ is a dense subspace of $[p],$ 
we conclude that $s \in \mathscr H_\kappa /[p].$

Assume now that $p, q, r$ are homogeneous polynomials with $\deg\, r < \deg\, p$. 
If $\deg\,q < \deg\,p - \deg \, r,$ then applying
the fact in the last paragraph to $s \in \{r, qr\}$, we conclude that 
$r, qr \in \mathscr H_\kappa /[p],$ and hence by part (ii), $q(\mathscr T_z)r = P^\perp (qr)=qr.$ 
Finally, if $q$ is such that $p=qr,$ then $P^\perp (qr)=P^\perp (p)=0.$

(v) This follows by applying part (iv) to $q(z)=z_j$ for $j=1, \ldots, d$. 
\end{proof}

The following lemma is based on the decomposition of a homogeneous polynomial in two variables (see Proposition\,A), a decomposition that fails in the case of three variables (see Remark\,A).
\begin{lemma} \label{poly-calculus-new} 
    Let $p, q \in \mathbb C[z_1, z_2]$ be nonconstant homogeneous polynomials, and let $\mathscr H_\kappa$ be a $\mathbb T^2$-invariant Hilbert module. Then the following statements are true$:$
\begin{enumerate}
\item[$(\mathrm{i})$] if $p \in [q],$ then $q$ is a factor of $p,$ 
\item[$(\mathrm{ii})$] $[p]=[q]$ if and only if $p$ is a constant multiple of $q,$
\item[$(\mathrm{iii})$] if $q(\mathscr T_{z,[p]}) = 0,$ then $\deg \, p \Le \deg \, q.$ 
\end{enumerate}
\end{lemma}
\begin{proof}
(i) Assume that $p \in [q].$ 
Thus, there exists a sequence of polynomials $\{f_n\}_{n \Ge 0}$ in $\mathscr H_\kappa$ such that $\{qf_n\}_{n \Ge 0}$ converges to $p$ in $\mathscr H_\kappa.$ This, combined with 
\beq \label{rkp}
\inp{f}{\kappa(\cdot, w)} = f(w), \quad f \in \mathscr H_\kappa, ~w \in \Omega,
\eeq
and the joint continuity of $\kappa$ (see \eqref{kappa-cts}), implies that $qf_n$ converges to $p$ compactly on $\Omega.$
By the Weierstrass convergence theorem (see \cite[Theorem~1.7.1]{JP2008}), 
\beq \label{WCT-a}
\mbox{$\{D^\alpha(qf_n)\}_{n \Ge 0}$ converges compactly to $D^\alpha(p)$ for all $\alpha \in \mathbb Z^2_+,$} 
\eeq
where $D^\alpha = \frac{\partial^{|\alpha|}}{\partial z^{\alpha_1}_1 \partial z^{\alpha_2}_2},$ $\alpha = (\alpha_1, \alpha_2) \in \mathbb Z^2_+.$  
Thus, $Z(q) \subseteq Z(p).$ By Proposition\,A, there exist scalars $a, b \in \mathbb C,$ not both zero, such that
\beqn
p(z)=(az_1-bz_2)^k r(z), \quad q(z)=(az_1-bz_2)^l s(z), 
\eeqn
where $k, l \in \mathbb Z_+$ and $r, s \in \mathbb C[z_1, z_2]$ are such that $Z(r)\cap Z(az_1- bz_2) =\{0\}$ and $Z(s)\cap Z(az_1- bz_2) = \{0\}$. Introduce new coordinates
\beqn
w_1:=az_1-bz_2, ~w_2:=\begin{cases} z_2 & \mbox{if~}a \neq 0, \\
z_1 & \mbox{if}~b \neq 0.
\end{cases}
\eeqn  
With these new conformal coordinates, near the origin, $p$ and $q$ takes the form
\beqn
p(w)=w^k_1 r(w), \quad q(w)=w^l_1 s(w), 
\eeqn
where $r(w)$ and $s(w)$ do not vanish near $(0, w_2)$ for $w_2 \neq 0.$
By \eqref{WCT-a}, 
\beq \label{cpt-cgn}
\Big\{\frac{\partial^m(w^l_1 s f_n)}{\partial w^m_1}\Big\}_{n \Ge 0}~\mbox{converges compactly to}~\frac{\partial^m (w^k_1 r)}{\partial w^m_1}~\mbox{for any integer}~m \Ge 1. 
\eeq
By the general Leibniz rule, 
\beq \label{GLR}
\frac{\partial^{k} (w^k_1 r)}{\partial w^{k}_1}=
\sum_{j=0}^{k}\binom{k}{j}\frac{\partial^{k-j}w^k_1}{\partial w^{k-j}_1} \frac{\partial^j r}{\partial w^j_1}=k!r + f 
\eeq
for some $f$ vanishing near $(0, w_2),$ $w_2 \neq 0.$ Thus, if $k < l,$ then by another application of the general Leibniz rule,  
\beqn
\lim_{n \rar \infty} \frac{\partial^{k}(w^l_1 s f_n)}{\partial w^{k}_1}(0, w_2)=0 \neq 
k!r(0, w_2) \overset{\eqref{GLR}}=\frac{\partial^{k} (w^k_1 r)}{\partial w^{k}_1}(0, w_2), \quad w_2 \neq 0,
\eeqn
since $r(0, w_2) \neq 0$ for $w_2 \neq 0$. This contradiction to \eqref{cpt-cgn} shows that $k \Ge l.$ Repeating this argument with every distinct factor of $q,$ we deduce that $q$ is a factor of $p.$ 

(ii) If $p$ is a constant multiple of $q$, then $[p]=[q].$ Now, if $[p]=[q],$ then $p \in [q]$ and $q \in [p],$ and hence by  (i) (after interchanging the roles of $p$ and $q$), $p$ is a constant multiple of $q$.  

(iii) Suppose that $q(\mathscr T_{z,[p]})=0.$ By Lemma~\ref{poly-calculus}(ii),
	\beqn
	P^\perp_{[p]}(q) = q(\mathscr T_{z, [p]})P^\perp_{[p]}(1)=0,
	\eeqn which implies that $q \in [p].$ By (i), $p$ is a factor of $q,$ and hence $\deg \, p \Le \deg \, q$.
\end{proof}

The following result proves the equivalence between conditions (ii) and (iii) in Theorems~\ref{Hardy-q-sub} and~\ref{Hardy-ball-q-sub}.

\begin{proposition} \label{lemma-main-thm}
Let $p$ be a nonconstant homogeneous polynomial in $\mathbb C[z_1, z_2],$
and let $\mathscr H_\kappa$ be a $\mathbb T^2$-invariant Hilbert module. Then the following are equivalent$:$
\begin{enumerate}
    \item[$\mathrm{(i)}$] $\deg \, p=1$,
    \item[$\mathrm{(ii)}$] there exist $a, b \in \mathbb C,$ not both zero, such that $a\mathscr T_{z_1, [p]}=b\mathscr T_{z_2, [p]}.$
\end{enumerate}
\end{proposition}
\begin{proof}
(i) $\Rightarrow$ (ii): This follows from Lemma~\ref{poly-calculus}(i).
	
(ii) $\Rightarrow$ (i): By assumption, $q(\mathscr T_{z, [p]})=0,$ where
$q(z_1,z_2)=az_1-bz_2.$ Hence, by
Lemma~\ref{poly-calculus-new}(iii), $\deg \,p \Le \deg \, q=1.$ Since $p$ is nonconstant, $\deg\,p=1.$
\end{proof}

Although only some special cases of Lemma~\ref{coro-generators} are required for the proofs of the main results, we include it for later use in Section~\ref{S5}. Certain special cases of this fact have appeared previously in the literature (see \cite[Section~4.1]{BM2001}, \cite{CM2020, DM1993}, \cite[Section~7.4]{DM2008}, and \cite[Proof of Corollary~2.4]{GW2007}). 

\begin{lemma} \label{coro-generators}
For nonnegative integers $r, s$ and $a \in \mathbb C,$ 
consider the homogeneous polynomial $p(z_1,z_2)= z_1^{r}z_2^s(z_1-az_2).$ Let $\mathscr H_\kappa$ be a $\mathbb T^2$-invariant Hilbert module.  
Consider the collection $\mathcal B$ of homogeneous polynomials given by
    \beq
\displaystyle p^{(j)}_n(z_1, z_2) = z^n_1 z^{j}_2, \quad 0 \Le j \Le s-1, ~n \Ge 0,~ s \Ge 1, \notag\\ \notag
\displaystyle q^{(j)}_n(z_1, z_2) =z^{j}_1 z^n_2, \quad 0 \Le j \Le r-1, ~n \Ge s,~ r \Ge 1, \\ \label{exp-q-r-n}
\displaystyle q^{(r)}_n(z_1, z_2) = \sum_{k=r}^{n-s}\overline{a}^{k-r} \frac{z^k_1 z^{n-k}_2}{\|z^k_1 z^{n-k}_2\|^2}, \quad n\Ge r+s;
 \eeq
if $r=0$ $($resp., $s=0)$, then the corresponding polynomial $q^{(j)}_n$ $($resp., $p^{(j)}_n)$ is excluded from $\mathcal B.$
Then $\mathcal B$ forms an orthogonal basis for the quotient module $\mathscr H_\kappa /[p].$
\end{lemma}
\begin{proof} 
	We first show that $\mathcal{B}$ is orthogonal to $[p]$. Since polynomials are dense in $\mathscr
	H_\kappa$ (see ({\tiny C4})) it suffices to check that each $p^{(j)}_n$ and $q^{(j)}_n$ is orthogonal to $z^{\alpha}p$ for every $\alpha \in \mathbb Z^2_+$. Since the powers of $z_2$ in $z^{\alpha}p$ is at least $s,$ by \eqref{onbasis}, $\inp{p^{(j)}_n}{z^{\alpha}p}=0$ 
for every $0 \Le j \Le s-1$ and $n \Ge 0.$ 
Similarly, since the powers of $z_1$ in $z^{\alpha}p$ is at least $r,$ $\inp{q^{(j)}_n}{z^{\alpha}p}=0$ for any $0 \Le j \Le r-1$ and $n \Ge s.$ 
Since $q^{(r)}_n$ and $z^{\alpha}p$ are homogeneous polynomials of degree $n$ and $|\alpha|+r+s+1,$ $\inp{q^{(r)}_n}{z^{\alpha}p}=0$ provided $n \neq |\alpha|+r+s+1.$ If $n = |\alpha|+r+s+1$ for some $\alpha \in \mathbb Z^2_+,$ then 
	\beqn
	\inp{q^{(r)}_n}{z^\alpha p} &=& \inp{q^{(r)}_n}{z_1^{\alpha_1+r+1} z_2^{\alpha_2+s} - az_1^{\alpha_1+r} z_2^{\alpha_2+s+1}} \\
&\overset{\eqref{exp-q-r-n}}=& \sum_{k=r}^{n-s}\frac{\overline{a}^{k-r}}{\|z^k_1 z^{n-k}_2\|^2} \, \inp{z^k_1 z^{n-k}_2}{z_1^{\alpha_1+r+1} z_2^{\alpha_2+s}-az_1^{\alpha_1+r} z_2^{\alpha_2+s+1}}\\
&\overset{\eqref{onbasis}}=& 0.
	\eeqn
This shows that $\mathcal B \subseteq \mathscr H_\kappa /[p].$

To see that $\mathscr H_\kappa \ominus \mathcal B$ is contained in $[p],$ let $g(z)=\sum_{\alpha \in \mathbb Z^2_+}\hat{g}_{\alpha} z^\alpha \in \mathscr H_\kappa$ belong to $\mathscr H_\kappa \ominus \mathcal B.$ Note that 
	\beq \label{formula-g-hat-1}
	&& \hat{g}_{n,j} = 0, \quad 0 \Le j \Le s-1, ~n \Ge 0, ~s \Ge 1, \\
\label{formula-g-hat-2}
&&	\hat{g}_{j,n}= 0, \quad 0 \Le j \Le r-1, ~n \Ge s, ~r \Ge 1, \\ \label{formula-g-hat-3}
&& 	
\sum_{k=r}^{n-s} {a}^{k-r} \hat{g}_{k, n-k} =0, \quad n\Ge r+s.
	\eeq
It follows that for every integer $m \Ge r+s,$  
	\beqn
	 \sum_{n=0}^{m}\sum_{k=0}^{n}\hat{g}_{k,n-k} z_1^k z_2^{n-k} 
	&\overset{\eqref{formula-g-hat-1}}{=}& z_2^s \sum_{n=s}^{m}\sum_{k=0}^{n-s}\hat{g}_{k,n-k} z_1^k z_2^{n-k-s} \\
	&\overset{\eqref{formula-g-hat-2}}{=}& z_1^{r}z_2^s  \sum_{n=r+s}^{m}\sum_{k=r}^{n-s}\hat{g}_{k,n-k} z_1^{k-r} z_2^{n-k-s} \\
	& \overset{\eqref{formula-g-hat-3}}=& 
 z_1^{r}z_2^s  \sum_{n=r+s}^{m}\sum_{k=r+1}^{n-s}\hat{g}_{k,n-k} (z_1^{k-r} - {a}^{k-r}z_2^{k-r})z_2^{n-k-s}.
	\eeqn
Since	$z_1^{l} - {a}^{l}z_2^{l},$ $l \Ge 1,$ is divisible by $z_1-az_2$ in $\mathbb C[z_1, z_2],$ the partial sum of $g$ belongs to $[p].$ It follows that $g \in [p],$ which completes the proof.
\end{proof}

Recall that $\{\gamma_\alpha\}_{\alpha \in \mathbb Z^d_+}$ is a {\it Stieltjes moment multi-sequence} (or a {\it Stieltjes moment sequence} when $d = 1$) if there exists a positive Borel measure on $\mathbb R^d_+$ (called a {\it representing measure} of $\{\gamma_\alpha\}_{\alpha \in \mathbb Z^d_+}$) such that
\beqn
\gamma_\alpha = \int_{\mathbb R^d_+} t^\alpha \, \D\mu(t), \quad \alpha \in \mathbb Z^d_+. 
\eeqn
\begin{remark} \label{rmk-exam-used}
For $a \in [0, 1),$ consider the sequence $\big\{\frac{1}{1-a^{n+1}}\big\}_{n \Ge 0}.$ 
Note that $\mu := \sum_{k=0}^{\infty}a^k \delta_{a^{k}}$ defines a finite Borel measure on $[0, 1]$ and satisfies 
\beqn
\frac{1}{1-a^{n+1}}= \sum_{k=0}^{\infty}a^{k(n+1)}=
\int_{[0, 1]} t^n \,\D\mu, \quad n \Ge 0.
\eeqn
Thus $\big\{\frac{1}{1-a^{n+1}}\big\}_{n \Ge 0}$ is a Stieltjes moment sequence. \eof
\end{remark}
The next lemma is key to deducing the subnormality of the quotient module $\mathscr H_\kappa/[p]$ when $\deg \,p=1$. 
\begin{lemma} \label{lem-sub-linear}
Let $p \in \mathbb C[z_1, z_2],$ and let $\mathscr H_\kappa$ be a $\mathbb T^2$-invariant Hilbert module. Then the following statements are true$:$ 
\begin{enumerate}
\item[$(\mathrm{i})$] if $p(z)= z_j,$ then $\mathscr T_{z_j}=0,$ $\mathscr M_{z_i}(\mathscr H_\kappa/[p]) \subseteq \mathscr H_\kappa/[p]$ and $\mathscr T_{z_i}=\mathscr M_{z_i}|_{\mathscr H_\kappa/[p]}$ for $1 \Le i \neq j \Le 2,$  
\item[$(\mathrm{ii})$] if $p(z)=z_1-az_2$ for some $a \in \mathbb C \backslash \{0\},$ then 
$\mathscr T_{z_1} = a \mathscr T_{z_2}$ and $\mathscr T_{z_2}$ is the unilateral weighted shift with respect to the orthonormal basis $\{\frac{q_n}{\|q_{n}\|}\}_{n \Ge 0}$ given by $\mathscr T_{z_2}(\frac{q_n}{\|q_{n}\|})=w_n \frac{q_{n+1}}{\|q_{n+1}\|},$ where
\beq 
\label{qn-expression}
q_n(z_1, z_2) &:=& \sum_{k=0}^{n}\overline{a}^{k} \frac{z^k_1 z^{n-k}_2}{\|z^k_1 z^{n-k}_2\|^2}, \quad n \Ge 0, \\
\label{weights}
w_n &:=& \frac{\|q_{n}\|}{\|q_{n+1}\|}, \quad n \Ge 0,
\eeq
\item[$(\mathrm{iii})$] if $p(z)=z_1-az_2$ for some $a \in \mathbb C \backslash \{0\},$ then the commuting pair $\mathscr T_z$ on $\mathscr H_\kappa/[p]$ is subnormal if and only if 
\beq \label{stieljes-m-s}
\left\{\Big(\displaystyle \sum_{k=0}^{n} \frac{{|a|}^{2k}}{\|z^k_1 z^{n-k}_2\|^2}\Big)^{-1}\right\}_{n \Ge 0}~\mbox{is a Stieltjes moment sequence,}
\eeq
\item[$(\mathrm{iv})$] if $p(z_1, z_2)=z_1-z_2,$ and $\mathscr H_{\kappa_j},$ $j=1, 2$ are $\mathbb T$-invariant Hilbert modules on the unit disc $\mathbb D,$ then the quotient module $\widehat{\mathscr H_{\kappa_1}\otimes\mathscr H_{\kappa_2}}/[p]$ is subnormal if and only if the module tensor product
 $\mathscr H_{\kappa_1} \otimes_{\mathbb C[z]} \mathscr H_{\kappa_2}$ is subnormal.
\end{enumerate}
\end{lemma} 
\begin{proof}
(i) If $p(z_1, z_2)= z_1,$ then by Lemma~\ref{coro-generators} (applied to $r=s=0$ and $a=0$), $\{z^n_2\}_{n \Ge 0}$ is an orthogonal basis for $\mathscr H_\kappa/[p],$ and hence $\mathscr T_{z_1}=0,$ $\mathscr H_\kappa/[p]$ is invariant under $\mathscr M_{z_2}$ and $\mathscr T_{z_2}=\mathscr M_{z_2}|_{\mathscr H_\kappa/[p]}.$ The proof is similar in case $p(z_1, z_2)= z_2.$ 

(ii) Suppose that $p(z_1, z_2)=z_1-az_2$ for some $a \in \mathbb C\backslash \{0\}.$ By Lemma~\ref{poly-calculus}(i), $\mathscr T_{z_1}= a\mathscr T_{z_2}$.
Also, by Lemma~\ref{coro-generators} (applied to $r=s=0$), $\{\frac{q_n}{\|q_{n}\|}\}_{n \Ge 0}$ is an orthonormal basis for $\mathscr H_\kappa/[p]$ (see \eqref{qn-expression}). 
By \eqref{onbasis}, 
$\{z^{\alpha}\}_{\alpha \Ge 0}$ is orthogonal in $\mathscr H_\kappa,$ and hence 
$\inp{z_2q_n}{q_m}=0$ for every $m \neq n+1$ and $\inp{z_2 q_n}{q_{n+1}}=\|q_{n}\|^2.$ It follows that 
\beqn
    \mathscr T_{z_2}(q_n)=  P^\perp\left(z_2 q_n\right)
 = \inp{z_2 q_n}{q_{n+1}} \frac{q_{n+1}}{\|q_{n+1}\|^2}
 =  \frac{\|q_{n}\|^2}{\|q_{n+1}\|^2}\, q_{n+1}, \quad n \Ge 0.
\eeqn
Thus $\mathscr T_{z_2}$ is a unilateral weighted shift with weights $\{w_n\}_{n \Ge 0}$ given by \eqref{weights}.

(iii) Note that
\beqn
\prod_{j=0}^{n-1}w^2_j = \frac{\|q_0\|^2}{\|q_n\|^2}=\frac{1}{\|1\|^2}\frac{1}{\displaystyle \sum_{k=0}^{n} \frac{{|a|}^{2k}}{\|z^k_1 z^{n-k}_2\|^2}}, \quad n \Ge 0.
\eeqn
By (ii) and \cite[Theorem~4]{SS1989}, $\mathscr T_{z_2}$ is subnormal if and only if \eqref{stieljes-m-s} holds.
If $\mathscr T_z$ is subnormal, then so $\mathscr T_{z_2}$, and hence \eqref{stieljes-m-s} holds. Conversely, if \eqref{stieljes-m-s} holds and $N$ is a normal extension of $\mathscr T_{z_2}$, then 
since $\mathscr T_{z_1}= a\mathscr T_{z_2}$ (see (ii)), $\mathscr T_z$ is subnormal with normal extension $(aN, N).$

(iv) Note that $\kappa_1 \otimes \kappa_2$ is given by
\beqn
\kappa_1 \otimes \kappa_2((z_1, z_2), (w_1, w_2)) &=&
\sum_{m=0}^{\infty}\frac{z^m_1\overline{w}^m_1}{\|z^m\|^2_{\mathscr H_{\kappa_1}}}\sum_{n=0}^{\infty}\frac{z^n_2\overline{w}^n_2}{\|z^n\|^2_{\mathscr H_{\kappa_2}}} \\
&=& \sum_{\alpha \in \mathbb Z^2_+}\frac{z^{\alpha_1}_1 z^{\alpha_2}_2 \,\overline{w}^{\alpha_1}_1\overline{w}^{\alpha_2}_2}{\|z^{\alpha_1}\|^2_{\mathscr H_{\kappa_1}}\|z^{\alpha_2}\|^2_{\mathscr H_{\kappa_2}}}, \quad (z_1, z_2), (w_1, w_2) \in \mathbb D^2.
\eeqn 
By \eqref{onbasis}, $\Big\{\frac{z^{\alpha_1}_1 z^{\alpha_2}_2}{\|z^{\alpha_1}\|_{\mathscr H_{\kappa_1}}\|z^{\alpha_2}\|_{\mathscr H_{\kappa_2}}}\Big\}_{\alpha \in \mathbb Z^2_+}$ is an orthonormal basis for $\widehat{\mathscr H_{\kappa_1} \otimes \mathscr H_{\kappa_2}}.$ Hence,   
\beqn
\|z^{\alpha_1}_1 z^{\alpha_2}_2\|_{\widehat{\mathscr H_{\kappa_1}\otimes\mathscr H_{\kappa_2}}}= \|z^{\alpha_1}\|_{\mathscr H_{\kappa_1}}\|z^{\alpha_2}\|_{\mathscr H_{\kappa_2}}, \quad \alpha=(\alpha_1, \alpha_2) \in \mathbb Z^2_+.
\eeqn
This, combined with the part (iii) (for $a=1$), implies that $\mathscr H_\kappa/[p]$ is subnormal if and only if 
\beqn
\left\{\Big(\displaystyle \sum_{k=0}^{n} \frac{1}{\|z^{k}\|^2_{\mathscr H_{\kappa_1}}}\frac{1}{\|z^{n-k}\|^2_{\mathscr H_{\kappa_2}}}\Big)^{-1}\right\}_{n \Ge 0}~\mbox{is a Stieltjes moment sequence.}
\eeqn
As noted in the discussion prior to \cite[Eq~(1.2)]{AC2017}, this holds if and only if the module tensor product $\mathscr H_{\kappa_1} \otimes_{\mathbb C[z]} \mathscr H_{\kappa_2}$ is a subnormal Hilbert module, which yields (iv). 
\end{proof}

Recall that $H^2(\mathbb D^2)$ is isomorphic to $\widehat{\mathscr H_{\kappa} \otimes \mathscr H_{\kappa}}$, where $\kappa$ denotes the Cauchy kernel of $H^2(\mathbb D)$. Thus, the above result relates the subnormality of $H^2(\mathbb D^2)/[z_1-z_2]$ to that of $\mathscr H_{\kappa^2},$ that is, the {\it Bergman module} on $\mathbb D.$ This has been noted in \cite[Section~5]{DM1993} (see also \cite[Corollary~5.17]{PR2016}).


For a commuting $d$-tuple $T=(T_1, \ldots, T_d)$ on $\mathcal H$, let $$Q_T(X)=\sum_{j=1}^dT^*_jXT_j, \quad X \in \mathcal B(\mathcal H).$$ 
We set
$Q^0_T(I)=I,$ and define inductively
$Q^n_T(I)=Q_T\big(Q^{n-1}_T(I)\big)$ for $n \Ge 1$. It is easy to see that
\beq \label{identity-sp-iso}
Q^n_T(I)=\sum_{\underset{|\alpha|=n}{\alpha \in \mathbb Z^d_+}}\frac{n!}{\alpha!}T^{*\alpha}T^{\alpha}, \quad n \in \mathbb Z_+.
\eeq 
Following \cite{GR2006}, we say that $T$ is an {\it $m$-isometry}, where $m \Ge 1,$ if 
$$\sum_{j=0}^m(-1)^j \binom{m}{j}Q^j_T(I) = 0.$$ For simplicity, we refer to a $1$-isometry simply as an {\it isometry}. We say that $T$ is a {\it contraction} (resp., a {\it toral isometry}) if $T^*_jT_j \Le I$ (resp., $T^*_{j}T_{j} = I$) for $j=1, \ldots, d.$

The following Hilbert modules are central to this paper:
\begin{example} 
\label{ex-Hardy}
The {\it Hardy module $H^2(\mathbb D^d)$} is the Hilbert space of analytic functions $f(z)=\sum_{\alpha \Ge 0}\hat{f}_{\alpha}z^\alpha$ on $\mathbb{D}^d$ with square-summable Fourier coefficients $\{\hat{f}_{\alpha}\}_{\alpha \Ge 0}.$ 
The reproducing kernel of $H^2(\mathbb D^d)$, called the {\it Cauchy kernel}, is given by
\beqn
\kappa(z, w)=\frac{1}{\prod_{j=1}^d(1-z_j \overline{w}_j)}, \quad z=(z_1, \ldots, z_d), \, w=(w_1, \ldots, w_d) \in \mathbb D^d.
\eeqn
Note that $H^2(\mathbb D^d)$ is a $\mathbb T^d$-invariant Hilbert module satisfying 
\beq
\label{ip-Hardy-pdisc}
\inp{z^\alpha}{z^\beta}_{H^2(\mathbb D^d)}=\begin{cases} 1 & \mbox{if}~\alpha = \beta \in \mathbb Z^d_+, \\
0 & \mbox{if}~\alpha \neq \beta \in \mathbb Z^d_+.
\end{cases}
\eeq 
It follows that the multiplication $d$-tuple $\mathscr M_z$ on $H^2(\mathbb D^d)$ is a toral isometry.

The {\it Hardy module $H^2(\mathbb B^d)$} is the Hilbert space of analytic functions $f(z)=\sum_{\alpha \Ge 0}\hat{f}_{\alpha}z^\alpha$ on $\mathbb{B}^d$ for which 
\beqn
\sum_{\alpha \Ge 0}|\hat{f}_\alpha|^2 \frac{\alpha!(d-1)!}{(|\alpha|+d-1)!} < \infty.
\eeqn 
The reproducing kernel of $H^2(\mathbb B^d),$ called the {\it Szeg$\ddot{o}$ kernel}, is given by
$\kappa(z, w)=\frac{1}{(1-\inp{z}{w})^d},$ $z, w \in \mathbb B^d.$
Note that $H^2(\mathbb B^d)$ is a $\mathcal U_d$-invariant Hilbert module satisfying
\beq
\label{ip-Hardy-ball}
\inp{z^\alpha}{z^\beta}_{H^2(\mathbb B^d)}=
\begin{cases} \frac{\alpha!(d-1)!}{(|\alpha|+d-1)!} & \mbox{if}~\alpha = \beta \in \mathbb Z^d_+, \\
0 & \mbox{if}~\alpha \neq \beta \in \mathbb Z^d_+.
\end{cases}
\eeq 
It follows that the multiplication $d$-tuple $\mathscr M_z$ on $H^2(\mathbb B^d)$ is an isometry.

The {\it Drury–Arveson module $H^2_d$} is the Hilbert space of analytic functions $f(z)=\sum_{\alpha \Ge 0}\hat{f}_{\alpha}z^\alpha$ on $\mathbb{B}^d$ for which 
\beqn
\sum_{\alpha \Ge 0}|\hat{f}_\alpha|^2 \frac{\alpha!}{|\alpha|!} < \infty.
\eeqn 
The reproducing kernel of $H^2_d,$ called the {\it Drury–Arveson kernel}, is given by
$\kappa(z, w)=\frac{1}{1-\inp{z}{w}},$ $z, w \in \mathbb B^d.$
Note that $H^2_d$ is a $\mathcal U_d$-invariant Hilbert module satisfying
\beq
\label{ip-Arveson}
\inp{z^\alpha}{z^\beta}_{H^2_d}=
\begin{cases} \frac{\alpha!}{|\alpha|!} & \mbox{if}~\alpha = \beta \in \mathbb Z^d_+, \\
0 & \mbox{if}~\alpha \neq \beta \in \mathbb Z^d_+.
\end{cases}
\eeq 
It follows from \cite[Lemma~4.3]{GR2006} that the multiplication $d$-tuple $\mathscr M_z$ on $H^2_d$ is a $d$-isometry. 
\hfill $\diamondsuit$
\end{example}


\begin{lemma} \label{Taylor-quotient}
Let $T$ be a commuting $d$-tuple on $\mathcal H$ and let $\mathcal M$ be a closed subspace of $\mathcal H$ that is invariant under $T.$ If $\sigma(T) \subseteq \overline{\mathbb B^d},$ then $\sigma(T|_{\mathcal M}) \subseteq \overline{\mathbb B^d}$.
\end{lemma}
\begin{proof} Recall from \cite[Theorem 1]{M-S} and \cite[Theorem 1]{C-Z} that the
geometric spectral
radius $r(S):=\sup \{\|z\|_2 : z \in \sigma(T)\}$ of a commuting $d$-tuple $S$ is given by 
\beq \label{spec-rad-f}
r(S)=\lim_{k \rar \infty}
\big\|Q_S^k(I)\big\|^{\frac{1}{2k}}.
\eeq
For any integer $k \Ge 1,$ note that 
\beqn
\big\|Q^k_{T|_{\mathcal M}}(I)\big\|=\sup_{x \in \mathcal M, \|x\|=1} \inp{Q^k_{T|_{\mathcal M}}(I)x}{x} \overset{\eqref{identity-sp-iso}}\Le \sup_{x \in \mathcal H, \|x\|=1} \inp{Q^k_{T}(I)x}{x}=\big\|Q^k_T(I)\big\|. 
\eeqn
It follows from \eqref{spec-rad-f} that $r(T|_{\mathcal M}) \Le r(T).$ This yields the required implication. 
\end{proof}

Finally, we include a measure-theoretic fact for completeness (cf. \cite[p.~308]{R1988}).
\begin{lemma} \label{lem-support}
Let $\Omega$ be a bounded domain in $\mathbb C^d$ and let $\mu$ be a finite positive Borel measure on $\overline{\Omega}.$ 
Assume that $f : \overline{\Omega} \rar [0, \infty)$ is a continuous function such that $\int_{\overline{\Omega}}f(z) \, \D\mu(z)=0.$ Then, 
$\mu(\overline{\Omega} \setminus Z(f))=0$ and 
$\supp\,\mu \subseteq Z(f).$
\end{lemma}

\section{Proofs of the main results} \label{S4}

We begin with a proof of Proposition~\ref{rotation-q-sub}.
\begin{proof}[Proof of Proposition~\ref{rotation-q-sub}]
Assume that $p=r q^2$ for some $r, q \in \mathbb C[z_1, \ldots, z_d]$ and $q$ is nonconstant. Since $p$ is homogeneous, the factors $q$ and $r$ of $p$ are also homogeneous (see the proof of Proposition\,B). Since $q$ is nonconstant, $\deg\,r < \deg\,p.$ Hence, by Lemma~\ref{poly-calculus}(iv), $r \in \mathscr H_\kappa/[p].$ Suppose now that $\mathscr T_z$ is subnormal with minimal normal extension $N$. Then there exists a positive measure $\mu_{r}$ on the Taylor spectrum $\sigma(N)$ of $N$ such that 
\beqn
\inp{\mathscr T_{z}^{*\alpha} \mathscr T_{z}^{\beta}(r)}{r} =\int_{\sigma(N)} \overline{w}^\alpha w^\beta \, \D\mu_{r}(w), \quad 
\alpha, \beta \in \mathbb Z^d_+
\eeqn
(see \cite[Proof of Theorem~2.1]{CS1985}).
Since $q^2(\mathscr T_z)r=0$ (see Lemma~\ref{poly-calculus}(iv)),  
\beqn
    \int_{\sigma(N)} |q^2(w)|^2 \, \D\mu_{r}(w) = \| q^2(\mathscr T_z)r\|^2 =0.
\eeqn
It follows from Lemma~\ref{lem-support} that $\supp\,\mu_r \subseteq Z(q),$ and hence 
\beqn
\| q(\mathscr T_z)r\|^2=\int_{Z(q)} |q(w)|^2 \, \D\mu_r(w) = 0.
\eeqn
Once again, by Lemma~\ref{poly-calculus}(iv), $q(\mathscr T_z)r= qr,$ which is nonzero. This contradiction shows that $q$ must be constant, or equivalently, $p$ is square-free. 
\end{proof}

The subnormality of the quotient module $H^2(\mathbb D^d)/[p]$ implies that $p$ is linear. 
\begin{proposition} \label{polydisc-linear}
Let $p \in \mathbb C[z_1, \ldots, z_d]$ be a nonconstant homogeneous polynomial. 
If $H^2(\mathbb D^d)/[p]$ is subnormal, then 
$\deg\,p=1.$
\end{proposition}
\begin{proof} In view of Remark~\ref{rmk-one-v}, we may assume that $d \Ge 2.$ 
Assume that the quotient module $H^2(\mathbb D^d)/[p]$ is subnormal and $\deg\,p \Ge 2.$ 
By Lemma~\ref{poly-calculus}(iv), $1 \in H^2(\mathbb D^d)/[p].$
Since $\mathscr M_z$ is a contraction, it follows from \eqref{adjoint} that $\mathscr T_z$ is contractive. By \cite[Corollary~II.2.17]{Co1991}, the minimal normal extension $N$ of $\mathscr T_z$ is a contraction. 
Thus the Taylor spectrum of $N$ is contained in $\overline{\mathbb D}^d.$ 
Hence, there exists a positive measure $\mu$ on $\overline{\mathbb D}^d$ such that 
\beq \label{app-sp-thm-new}
\|\mathscr T^{\alpha}_z(1)\|^2 =\int_{\overline{\mathbb D}^d} |w^{\alpha}|^{2} \D\mu(w), \quad \alpha \in  \mathbb Z^d_+.
\eeq
Since $\deg\,p \Ge 2,$ this, combined with Lemma~\ref{poly-calculus}(v), implies that 
\beqn
    \int_{\overline{\mathbb D}^d} (1-|w_j|^2) \, \D\mu(w) &=& \|1\|^2 - \|\mathscr T_{z_j}(1)\|^2 \\
&=& 1 - \|z_j\|^2 \\
&\overset{\eqref{ip-Hardy-pdisc}}=& 0, \quad j=1, \ldots, d.
\eeqn 
Since $1-|w_j|^2$ is nonnegative on $\overline{\mathbb D}^d$ and $\mu$ is a positive measure, by Lemma~\ref{lem-support}, $\supp\,\mu \subseteq Z(1-|z_j|^2) \cap \overline{\mathbb D}^d$ for every $j=1, \ldots, d.$ It follows that $\supp \,\mu \subseteq \mathbb T^d.$
This, together with \eqref{app-sp-thm-new}, yields
\begin{align*}
\|\mathscr T^{\alpha}_z(1)\|^2 =\int_{\mathbb T^d} 1\, \D\mu(w)
    =1, \quad \alpha \in  \mathbb Z^d_+.
\end{align*}
However, by Lemma~\ref{poly-calculus}(ii), $\mathscr T^{\alpha}_z(1)=P^{\perp}(z^\alpha),$ and hence $\|P^{\perp}(z^\alpha)\|=1$ for every $\alpha \in \mathbb Z^d_+.$ It follows that
\beqn
   \|P(z^\alpha)\|^2 = \|z^\alpha\|^2 - \|P^{\perp}(z^\alpha)\|^2 \overset{\eqref{ip-Hardy-pdisc}}=0, \quad \alpha \in  \mathbb Z^d_+.
\eeqn
This forces $P (z^\alpha)=0$ for every 
$\alpha \in  \mathbb Z^d_+,$ forcing
$P(p)=0$ or $p \in H^2(\mathbb D^d)/[p],$ which is a contradiction.
\end{proof}

\begin{proof}[Proof of Theorem~\ref{Hardy-q-sub}]
(ii) $\Rightarrow$ (i): Since $\mathscr M_z$ is subnormal, by Lemma~\ref{lem-sub-linear}(i), $\mathscr T_z$ is subnormal provided either $p(z_1, z_2)=z_1$ or $p(z_1, z_2)=z_2.$ 
Suppose that $p(z_1, z_2)=z_1-az_2$ for some $a \in \mathbb C \backslash \{0\}.$ 
By Lemma~\ref{lem-sub-linear}(iii) and \eqref{ip-Hardy-pdisc}, it suffices to check that
\beq
\label{Stieljes-Hardy}
\Big\{\Big(\sum_{k=0}^{n} {|a|}^{2k}\Big)^{-1}\Big\}_{n \Ge 0}~\mbox{is a Stieltjes moment sequence.}
\eeq 
If $|a|=1,$ then this follows directly from the identity $\frac{1}{n+1}=\int_{0}^1 t^n\,\D t,$ $n \Ge 0.$
If $0 < |a| < 1,$ then 
\beqn
\frac{1}{\displaystyle \sum_{k=0}^{n} {|a|}^{2k}}= \frac{1-|a|^2}{1-|a|^{2(n+1)}}, \quad n \Ge 0,
\eeqn
and hence \eqref{Stieljes-Hardy} follows from Remark~\ref{rmk-exam-used}.  
If $|a| > 1,$ then 
\beqn
\frac{1}{\displaystyle \sum_{k=0}^{n}|a|^{2k}}=\frac{1}{\displaystyle \sum_{k=0}^{n}|a|^{-2k}}\frac{1}{|a|^{2n}}, \quad n \Ge 0.
\eeqn
Because the product of two Stieltjes moment sequences is a Stieltjes moment sequence (see \cite[Lemma 2.1]{BD2004}), \eqref{Stieljes-Hardy} holds in this case as well.

(i) $\Rightarrow$ (ii): This follows from Proposition~\ref{polydisc-linear}.

The equivalence (ii)$\Leftrightarrow$(iii) follows from Proposition~\ref{lemma-main-thm}.
\end{proof}
\begin{remark} \label{rmk-rep-measures}
Suppose that $p(z_1, z_2)=z_1-az_2$ for some $a \in \mathbb C \backslash \{0\}.$ Since $\mathscr T_{z_2}$ is a unilateral weighted shift with weights $\{w_n\}_{n \Ge 0}$ given by
\eqref{weights} and $\mathscr T_{z_1}= a\mathscr T_{z_2}$, it is easy to see that for $k, l, n \Ge 0,$
  \begin{align*}
 \gamma_{a}(k, l) :=    \Big\|\mathscr T_{z_1}^k \mathscr T_{z_2}^l\big(\frac{q_n}{\|q_n\|}\big)\Big\|^2
    = \begin{cases}
       \frac{n+1}{n+k+l+1}& |a|=1,\\  
      |a|^{2k}\frac{1- |a|^{2(n+1)}}{1- |a|^{2(n+k+l+1)}}& |a|\neq 1,
      \end{cases}
  \end{align*}
where $\{q_n\}_{n \Ge 0}$ is as given in \eqref{qn-expression}.
The representing measures of these Stieltjes moment bi-sequences can be computed as follows:
\begin{enumerate}
    \item[$\bullet$] $|a|=1$:  
Define $\varphi: [0,\,1] \to [0,\,1]^2$ by $\varphi(t)=(t, t)$ for $t \in [0, 1]$. 
Then,
    $$\gamma_{a}(k, l) = \int_{[0, 1]^2} t_1^{2k} t_2^{2l} \,2(n + 1)t^{2n+1}_1 \, \D\varphi_*\ell(t), \quad k, l \Ge 0.$$ 
    \item[$\bullet$] $|a|<1$: 
Note that for any $k, l \Ge 0,$
    \begin{align*}
\gamma_{a}(k, l) = \int_{[0, 1]^2} t_1^{2k} t_2^{2l} \big(1- |a|^{2(n+1)}\big)\sum_{m=0}^{\infty}|a|^{2m(n+1)}\,\D\delta_{(|a|^{m+1},\,|a|^{m})}(t). 
   \end{align*}   
\item[$\bullet$] $|a|>1$: Note that for any $k, l \Ge 0,$
    \begin{align*}
\gamma_{a}(k, l) =
\int_{[0, 1]^2} t_1^{2k} t_2^{2l} \,\frac{|a|^{2(n+1)}-1}{|a|^{2(n+1)}}\sum_{m=0}^{\infty}\frac{1}{|a|^{2m(n+1)}}\,\D\delta_{(|a|^{-m}, |a|^{-(m+1)})}(t).
\end{align*} 
\end{enumerate}
We leave the verification of these identities to the interested reader. \eof
\end{remark}

We need Proposition~\ref{sub-linear-u-invariant} in the proof of Theorem~\ref{Hardy-ball-q-sub}.
\begin{proof}[Proof of Proposition~\ref{sub-linear-u-invariant}]
Since $\mathscr M_z$ is subnormal, by Lemma~\ref{lem-sub-linear}(i), $\mathscr T_z$ is subnormal provided either $p(z_1, z_2)=z_1$ or $p(z_1, z_2)=z_2.$ Suppose now that $p(z_1, z_2)=z_1-az_2$ for some $a \in \mathbb C \backslash \{0\}.$ In view of Lemma~\ref{lem-sub-linear}, it suffices to check that \eqref{stieljes-m-s} holds. 
Since $\mathscr H_\kappa$ is $\mathcal U_2$-invariant, by \cite[Theorem~2.1]{CY2015}, there exists a sequence $\{\gamma_k\}_{k \Ge 0}$ of positive real numbers such that
\beqn
\|z^{\alpha}\| = \gamma_{|\alpha|}\sqrt{\frac{\alpha!}{(|\alpha|+1)!}}, \quad \alpha \in \mathbb Z^2_+.
\eeqn
It follows that
\beqn
\sum_{k=0}^{n} \frac{{|a|}^{2k}}{\|z^k_1 z^{n-k}_2\|^2} = \frac{1}{\gamma^2_{n}} (n+1)(1+|a|^2)^n, \quad n \Ge 0.
\eeqn
However, since $\mathscr H_\kappa$ is subnormal, by \cite[Theorem 5.3]{CY2015} and \cite[Theorem~4]{SS1989}, $\{\gamma^2_k\}_{k \Ge 0}$ is a Stieltjes moment sequence, and hence   
\beqn
\Big(\displaystyle \sum_{k=0}^{n} \frac{{|a|}^{2k}}{\|z^k_1 z^{n-k}_2\|^2}\Big)^{-1}= \gamma^2_{n}\,\frac{1}{n+1}\frac{1}{(1+|a|^2)^n}, \quad n \Ge 0,
\eeqn
being product of finitely many Stieljes moment sequences, is a Stieltjes moment sequence (see \cite[Lemma 2.1]{BD2004}). This completes the proof. 
\end{proof}
\begin{remark} \label{example-Arveson-sub-deg-1}
Assume that $\deg\,p=1.$ We show that $H^2_2/[p]$ is subnormal.
Since $\mathscr M_{z_j},$ $j=1, 2,$ is subnormal (see \cite[Proposition~4]{L1977}), by Lemma~\ref{lem-sub-linear}(i), $\mathscr T_z$ is subnormal provided either $p(z_1, z_2)=z_1$ or $p(z_1, z_2)=z_2.$ Suppose now that $p(z_1, z_2)=z_1-az_2$ for some $a \in \mathbb C \backslash \{0\}.$ In view of Lemma~\ref{lem-sub-linear}, it suffices to check that 
\eqref{stieljes-m-s} holds. To see this, note that 
\beqn
\Big(\sum_{k=0}^{n} \frac{{|a|}^{2k}}{\|z^k_1 z^{n-k}_2\|^2}\Big)^{-1}
\overset{\eqref{ip-Arveson}}=
 \Big(\sum_{k=0}^{n} \binom{n}{k}{|a|}^{2k}\Big)^{-1} = (1+|a|^2)^{-n}, \quad n \Ge 0.
\eeqn
This is a Stieltjes moment sequence with the representing measure $\delta_{(1+|a|^2)^{-1}}$, and hence the quotient module $H^2_2/[p]$ is subnormal. 
\eof
\end{remark}

The proof of Theorem~\ref{Hardy-ball-q-sub} further requires the following facts.
\begin{proposition} \label{spherical-m-iso-new}
Let $m$ be a positive integer, $p$ be a nonconstant homogeneous polynomial in $\mathbb C[z_1, \ldots, z_d],$ and let $\mathscr H_\kappa$ be a $\mathbb T^d$-invariant Hilbert module on $\mathbb B^d$. 
Assume that the quotient module $\mathscr H_\kappa/[p]$ is subnormal. 
Then $\deg\,p \Le m$ whenever one of the following conditions holds$:$
\begin{enumerate}  
\item[$\mathrm{(i)}$] $\mathscr M_z$ is an $m$-isometry, 
\item[$\mathrm{(ii)}$] $\sigma(\mathscr M_z) \subseteq \overline{\mathbb B^d}$ and 
\beq \label{weak-m-iso}
\displaystyle \sum_{n=0}^m (-1)^n\binom{m}{n}\|z^n_j\|^2=0, \quad j=1, \ldots, d.
\eeq 
\end{enumerate}  
\end{proposition}
\begin{proof} In view of Remark~\ref{rmk-one-v}, we may assume that $d \Ge 2.$ Since the Taylor spectrum of any $m$-isometry is contained in the closed unit ball $\overline{\mathbb B^d}$ (see the discussion following \cite[Lemma~3.2]{GR2006}), we may assume that 
$\sigma(\mathscr M_z) \subseteq \overline{\mathbb B^d}.$ 
Since $\sigma(S^*)=\{\overline{z} : z \in \sigma(S)\}$ for any commuting $d$-tuple $S$ on $\mathcal H,$ by Lemma~\ref{Taylor-quotient} (applied to $S=\mathscr M^*_z$ and $\mathcal M=\mathscr H_\kappa \ominus [p]$), we have $\sigma(\mathscr T_z) \subseteq \overline{\mathbb B^d}.$ 
By \cite[Theorem]{P1984}, the Taylor spectrum of the minimal normal extension $N$ of $\mathscr T_z$ is contained in $\overline{\mathbb B^d}.$ 
Hence, there exists a positive measure $\mu$ on $\overline{\mathbb B^d}$ such that 
\beq \label{app-sp-thm-new-ball}
\|\mathscr T^{\alpha}_z(1)\|^2 =\int_{\overline{\mathbb B^d}} |w^{\alpha}|^{2} \,\D\mu(w), \quad \alpha \in  \mathbb Z^d_+.
\eeq

Assume that (i) holds and that $\deg\,p \Ge m+1$.
By parts (ii) and (iv) of Lemma~\ref{poly-calculus}, $1 \in \mathscr H_\kappa/[p]$, and
\beq
\label{eq-action-1}
\mathscr T^{\alpha}_z(1)=P^{\perp}(z^\alpha), ~\alpha \in \mathbb Z^d_+,
\\ \label{eq-action-2}
P(z^{\alpha})=0, ~|\alpha| \Le m, 
\\ \label{eq-action-3}
\mathscr T^{\alpha}_z(1) = z^\alpha, ~|\alpha| \Le m.
\eeq
Note that \eqref{app-sp-thm-new-ball}, together with the multinomial theorem, implies that  
\beqn
    \int_{\overline{\mathbb B^d}} \Big(1-\sum_{j=1}^d|w_j|^2\Big)^m \,\D\mu(w) &=& \sum_{n=0}^m (-1)^n \binom{m}{n}\sum_{\underset{|\alpha|=n}{\alpha \in \mathbb Z^d_+}}\frac{n!}{\alpha!}\|\mathscr T^{\alpha}_z(1)\|^2 \\
&\overset{\eqref{eq-action-3}}=& \sum_{n=0}^m (-1)^n \binom{m}{n}\sum_{\underset{|\alpha|=n}{\alpha \in \mathbb Z^d_+}}\frac{n!}{\alpha!}\|z^{\alpha}\|^2, 
\eeqn
which is $0$, since $\mathscr M_z$ is an $m$-isometry.
Since $1-\sum_{j=1}^d|w_j|^2$ is nonnegative on $\overline{\mathbb B^d}$ and $\mu$ is a positive measure, by Lemma~\ref{lem-support}, 
\beq \label{support-sphere}
\supp \,\mu \subseteq \partial \mathbb B^d.
\eeq 
\allowdisplaybreaks
This, combined with \eqref{app-sp-thm-new-ball} and $\|f\|^2=\|Pf\|^2+\|P^{\perp}f\|^2,$ $f \in \mathscr H_\kappa$, implies that 
		\beqn
	&&	\int_{\overline{\mathbb B^d}} \Big(1-\sum_{j=1}^d|w_j|^2\Big)^{m+k} \, \D\mu(w) \\
		&=& \sum_{n=0}^{m+k} (-1)^n \binom{m+k}{n}\sum_{\underset{|\alpha|=n}{\alpha \in \mathbb Z^{d}_+}}\frac{n!}{\alpha!}\|\mathscr T^{\alpha}_z(1)\|^2 \\
		&\overset{\eqref{eq-action-1}}=& \sum_{n=0}^{m+k} (-1)^n \binom{m+k}{n} \sum_{\underset{|\alpha|=n}{\alpha \in \mathbb Z^{d}_+}}\frac{n!}{\alpha!}\|P^{\perp}(z^{\alpha})\|^2 \\
		&=& \sum_{n=0}^{m+k} (-1)^n \binom{m+k}{n} \sum_{\underset{|\alpha|=n}{\alpha \in \mathbb Z^{d}_+}}\frac{n!}{\alpha!}\big(\|z^{\alpha}\|^2- \|P z^{\alpha}\|^2\big)\\
		&\overset{(*) \,\&\, \eqref{eq-action-2}}=& -\sum_{n=m+1}^{m+k} (-1)^n \binom{m+k}{n} \sum_{\underset{|\alpha|=n}{\alpha \in \mathbb Z^{d}_+}}\frac{n!}{\alpha!}\|P z^{\alpha}\|^2, \quad k \Ge 1,
		\eeqn
		where ($*$) follows from the fact that an $m$-isometry is also an $l$-isometry for any integer $l > m.$ It now follows from \eqref{support-sphere} that 
\beq
\label{induction-sp-m-iso}
\sum_{n=m+1}^{m+k} (-1)^n \binom{m+k}{n} \sum_{\underset{|\alpha|=n}{\alpha \in \mathbb Z^{d}_+}}\frac{n!}{\alpha!}\|P z^{\alpha}\|^2=0, \quad k \Ge 1.
\eeq
We prove by a strong induction on $|\alpha| \Ge m+1$ that
\beq
\label{ind-hypo-sp-m-iso}
P(z^{\alpha}) =0,~|\alpha| \Ge m+1.
\eeq
Letting $k=1$ in \eqref{induction-sp-m-iso}, we obtain 
\beqn 
\sum_{\underset{|\alpha|=m+1}{\alpha \in \mathbb Z^{d}_+}}\frac{(m+1)!}{\alpha!}\,\|P(z^{\alpha})\|^2=0 ~\Rightarrow ~ P(z^{\alpha}) =0,~|\alpha|=m+1.
\eeqn
This gives \eqref{ind-hypo-sp-m-iso} for $|\alpha|=m+1.$ Now assume that \eqref{ind-hypo-sp-m-iso} for $m+1 \Le |\alpha| \Le m+l$ for some integer $l \Ge 1.$ Applying \eqref{induction-sp-m-iso} to $k=l+1,$ we obtain  
\beqn
&&(-1)^{m+l+1} \sum_{\underset{|\alpha|=m+l+1}{\alpha \in \mathbb Z^{d}_+}}\frac{(m+l+1)!}{\alpha!}\|P z^{\alpha}\|^2\\
&=&\sum_{n=m+1}^{m+l+1} (-1)^n \binom{m+l+1}{n}\sum_{\underset{|\alpha|=n}{\alpha \in \mathbb Z^{d}_+}}\frac{n!}{\alpha!}\|P z^{\alpha}\|^2=0,
\eeqn
which yields \eqref{ind-hypo-sp-m-iso} for $|\alpha|=m+l+1.$ 
This completes the proof of \eqref{ind-hypo-sp-m-iso}. Since $\deg\,p \Ge m+1,$ $P(p)=0$ or $p\in \mathscr H_\kappa/[p],$ which is a contradiction.


Assume that \eqref{weak-m-iso} holds and that $\deg\,p \Ge m+1$.   
By Lemma~\ref{poly-calculus}(iv), $1 \in \mathscr H_\kappa/[p]$ and
$\mathscr T^n_{z_j}(1) = z^{n}_j$ for $n=1, \ldots, m$ and $1 \Le j \Le d$. In particular, for any $j=1, \ldots, d,$
\beqn
\int_{\overline{\mathbb B^d}} (1 - |w_j|^2)^m \, \D \mu(w)
&\overset{\eqref{app-sp-thm-new-ball}}=&
\sum_{n=0}^m (-1)^n \binom{m}{n}\|\mathscr T^n_{z_j}(1)\|^2 \\
&=& \sum_{n=0}^m (-1)^n\binom{m}{n}\|z^n_j\|^2,
\eeqn
which is $0$ by the assumption (ii). It now follows from Lemma~\ref{lem-support} that
\[
\supp\,\mu
\subseteq
\bigcap_{j=1}^d Z(1 - |w_j|^2) \cap \overline{\mathbb B^d},
\]
which is an empty set since $d \Ge 2.$ However, support of any nonzero, compactly supported positive Borel measure is nonempty. 
This contradiction shows that $\deg\,p \Le m.$ 
\end{proof}
\begin{remark} 
Proposition~\ref{spherical-m-iso-new} applies to the kernel $\kappa_b$ given by
\beqn
\kappa_b(z, w)=\frac{1}{(1-\inp{z}{w})^{b}}, \quad z, w \in \mathbb B^d,
\eeqn 
whenever $b$ is an integer such that $1 \Le b \Le d.$ Indeed, by \cite[Lemma~4.3]{GR2006}, the multiplication $d$-tuple $\mathscr M_z$ on $\mathscr H_{\kappa_b}$ is a $(d-b+1)$-isometry, 
and hence $\sigma(\mathscr M_z) \subseteq \overline{\mathbb B^d}$. 
It follows that if the quotient module $\mathscr H_{\kappa_b}/[p]$ is subnormal, then  
$\deg\,p \Le d-b+1.$ In particular, letting $b=d$, we obtain
\beq \label{Hardy-ball-n-cond}
\mbox{the subnormality of $H^2(\mathbb B^d)/[p]$ implies that $\deg\,p \Le 1.$}
\eeq
Moreover, if $b=1$, then 
\beq \label{Arveson-ball-n-cond}
\mbox{the subnormality of $H^2_d/[p]$ implies that $\deg\,p \Le 1.$}
\eeq
Since \eqref{weak-m-iso} is satisfied for $m=1$ (see \eqref{ip-Arveson}), \eqref{Arveson-ball-n-cond} follows immediately.
\eof
\end{remark}

\begin{proof}[Proof of Theorem~\ref{Hardy-ball-q-sub}] 

To see the equivalence (i)$\Leftrightarrow$(ii) for $H^2(\mathbb B^2),$
note that the implication (ii) $\Rightarrow$ (i) follows from Proposition~\ref{sub-linear-u-invariant}, while the implication (i) $\Rightarrow$ (ii) follows from \eqref{Hardy-ball-n-cond}.

%


To see the equivalence (i)$\Leftrightarrow$(ii) for $H^2_2,$
note that the implication (ii) $\Rightarrow$ (i) follows from Remark~\ref{example-Arveson-sub-deg-1}, while the implication (i)$\Rightarrow$(ii) follows from \eqref{Arveson-ball-n-cond}.

In both cases, the equivalence (ii)$\Leftrightarrow$(iii) follows from Proposition~\ref{lemma-main-thm}. 
\end{proof}

As a consequence of the proof of Proposition~\ref{rotation-q-sub}, we obtain the following dichotomy for (necessarily non-homogeneous) submodules. 
\begin{corollary}
Let $p, q\in \mathbb C[z_1,z_2],$ and let $\mathscr H_\kappa$ be a $\mathbb T^2$-invariant Hilbert module. If the quotient module $\mathscr H_\kappa/[pq^2]$ is subnormal, then either $[p] = [pq^2]$ or $[pq] = [pq^2]$. 
\end{corollary}
\begin{proof} 
Suppose that 
$[p] \neq [pq^2].$ 
Since $[pq^2] \subseteq [p]$ and $[p] \neq [pq^2],$ we have $r:=P^\perp_{[pq^2]}(p) \neq 0$. Assume that $\mathscr H_\kappa/[pq^2]$ is subnormal. 
Thus, there exists a positive measure $\mu_{r}$ on a compact set $\sigma$ such that 
\beqn
\inp{\mathscr T_{z}^{*m} \mathscr T_{z}^{n}(r)}{r} =\int_{\sigma} \overline{w}^m w^n \, \D\mu_{r}(w), \quad m, n \in \mathbb Z^2_+.
\eeqn
 By Lemma~\ref{poly-calculus}(ii), 
\beqn
q^2(\mathscr T_z)r= q^2(\mathscr T_z)P^\perp_{[pq^2]}(p) = P^\perp_{[pq^2]}(q^2 p)=0, 
\eeqn
and consequently,  
\beqn
    \int_{\sigma} |q^2(w)|^2 d\mu_{r}(w) = \| q^2(\mathscr T_z)r\|^2 =0.
\eeqn
Since $|q^2|^2$ is nonnegative, 
by Lemma~\ref{lem-support}, $\mu_r(\sigma\backslash Z(q))=0,$ and hence
\beqn
\|q(\mathscr T_z)r\|^2 =\int_{Z(q)} |q(w)|^2 \, \D\mu_{r}(w) = 0.
\eeqn
However, by Lemma~\ref{poly-calculus}(ii), $P^\perp_{[pq^2]}(qp) =q(\mathscr T_z)r=0.$ This implies that $pq \in [pq^2]$. Since
$[pq^2] \subseteq [pq],$ we obtain $[pq]= [pq^2].$ 
\end{proof}

\section{Examples illustrating the main results}\label{S5}

We show below that Theorem~\ref{Hardy-q-sub} fails for $\mathbb T^2$-invariant Hilbert modules on $\mathbb D^2$.
To see this, consider the homogeneous polynomial $p(z_1, z_2)=z_1z_2,$ and  
note that, by Lemma~\ref{coro-generators} (applied to $r=0,$ $s=1,$ and $a=0$), 
\beq \label{ogb-z1z2}
\mbox{$\{z^m_1 : m \Ge 0\} \cup \{z^n_2 : n \Ge 1\}$ forms an orthogonal basis for $\mathscr H_\kappa/[z_1z_2].$}
\eeq
It is now straightforward to check that
\beq
\label{action-Tz-12}
 \left.
 \begin{array}{ccc}
\mathscr T_{z_1}(z^m_1)=z^{m+1}_1, ~
\mathscr T_{z_1}(z^n_2)=0, \\[4pt]
 \mathscr T_{z_2}(z^n_1)=0, ~
\mathscr T_{z_2}(z^m_2)=z^{m+1}_2,
\end{array}
\right\}
 \quad m \Ge 0, ~n \Ge 1.
\eeq
We next state a key fact that helps determine an appropriate kernel $\kappa$ for the desired example.

\begin{proposition} \label{prop-nonlinear}
Let $\mathscr H_\kappa$ be a $\mathbb T^2$-invariant Hilbert module. Assume that 
\beq
\label{rep-measure}
\|z_1\|^2 + \|z_2\|^2=\|1\|^2, ~\|z^n_j\|=\|z_j\|, ~j=1, 2, ~n \Ge 2.
\eeq
Then the quotient module $\mathscr H_\kappa/[z_1z_2]$ is subnormal. In particular, $\mathscr T_z$ on $\mathscr H_\kappa/[z_1z_2]$ is an isometry.
\end{proposition}
\begin{proof}
By \eqref{ogb-z1z2}, we may write $f \in \mathscr H_\kappa/[z_1z_2]$ as $\sum_{m = 0}^\infty a_m z^m_1 + \sum_{n = 1}^\infty b_n z^n_2$ for some scalars $a_m, b_n \in \mathbb C.$ Note that 
\beqn
&& \|f\|^2-\|\mathscr T_{z_1}f\|^2 - \|\mathscr T_{z_2}f\|^2 
\overset{\eqref{action-Tz-12}}= |a_0|^2(\|1\|^2-\|z_1\|^2-\|z_2\|^2) \\ &+& 
\sum_{n = 1}^\infty |a_n|^2(\|z^n_1\|^2 -\|z^{n+1}_1\|^2) + \sum_{n = 1}^\infty |b_n|^2(\|z^n_2\|^2-\|z^{n+1}_2\|^2), 
\eeqn
which is $0$ in view of \eqref{rep-measure}. Thus $\mathscr T_z$ is an isometry, and hence by \cite[Proposition~2]{A1990}, $\mathscr T_z$ is subnormal. 
\end{proof}

We now present an example of a subnormal quotient module $\mathscr H_\kappa/[p]$ when $\deg\,p = 2.$ 

\begin{example} \label{fail-linear-bidisc}
Consider the $\mathbb T^2$-invariant Hilbert module $\mathscr H_\kappa$ on the unit bidisc $\mathbb D^2$ associated with the reproducing kernel 
\beqn
\kappa(z, w)=\frac{2}{\prod_{j=1}^2(1-z_j \overline{w}_j)}-1, \quad z=(z_1, z_2), \, w=(w_1, w_2) \in \mathbb D^2.
\eeqn
It follows from \eqref{onbasis} that  
\beqn
\|z^\alpha\|^2 = \begin{cases} 1 & \mbox{if~}\alpha = 0, \\
\frac{1}{2} & \mbox{if~}\alpha \in \mathbb Z^2_+\backslash \{0\}.
\end{cases}
\eeqn
Note that $\mathscr H_\kappa$ is a subnormal Hilbert module. Indeed, 
\beqn
\|z^\alpha\|^2 = \int_{[0, 1]^2} t^{\alpha}\, \D(\delta_{(0, 0)}+\delta_{(1, 1)})/2, \quad \alpha \in \mathbb Z^2_+.
\eeqn
Since \eqref{rep-measure} is satisfied, by Proposition~\ref{prop-nonlinear}, $\mathscr H_\kappa/[z_1z_2]$ is subnormal. 
\hfill $\diamondsuit$
\end{example} 

Clearly, the Hilbert module in Example~\ref{fail-linear-bidisc} is not $\mathcal U_2$-invariant. This raises the question of whether a $\mathcal U_2$-invariant Hilbert module $\mathscr H_\kappa$ admits a subnormal quotient module $\mathscr H_\kappa/[p]$ when $\deg\,p = 2.$  The answer is affirmative.

\begin{example} \label{fail-linear-ball}
		Consider the $\mathcal{U}_2$-invariant Hilbert module $\mathscr H_\kappa$ on the unit ball $\mathbb B^2$ associated with the reproducing kernel 
		\beqn
		\kappa(z, w)=\frac{2}{(1-\inp{z}{w})^2}-1, \quad z, w \in \mathbb B^2.
		\eeqn
It follows from \eqref{onbasis} that  
		\beq \label{norm-u2}
		\|z^\alpha\|^2 = \begin{cases} 1 & \mbox{if~}\alpha = 0, \\
			\frac{\alpha!}{2(|\alpha|+1)!} & \mbox{if~}\alpha \in \mathbb Z^2_+\backslash \{0\}.
		\end{cases}
		\eeq
Then $\mathscr H_\kappa$ is a subnormal Hilbert module. Indeed, since $\mathscr M_z$ on $H^2(\mathbb B^2)$ is subnormal, $\{\|z^\alpha\|^2_{H^2(\mathbb B^2)}\}_{\alpha \in \mathbb Z^2_+}$ (see \eqref{ip-Hardy-ball}) is a Stieltjes moment sequence. If $\mu$ is the representing measure of $\{\|z^\alpha\|^2_{H^2(\mathbb B^2)}\}_{\alpha \in \mathbb Z^2_+}$, then for every $\alpha \in \mathbb Z^2_+,$
		\beqn
		\|z^\alpha\|^2_{\mathscr H_\kappa} = \frac{1}{2}\|z^{\alpha}\|^2_{H^2(\mathbb B^2)} + \frac{1}{2}\int_{[0, 1]^2} t^{\alpha}\, \D\delta_{(0, 0)}=
\int_{[0, 1]^2} t^{\alpha}\, \D\big(\mu+\delta_{(0, 0)}\big)/2.
		\eeqn
It is easy to see using \eqref{action-Tz-12} and \eqref{norm-u2} that for $\alpha=(\alpha_1, \alpha_2) \in \mathbb Z^2_+,$ 
		\beqn
		\|\mathscr T_{z}^\alpha(1)\|^2=\begin{cases}
			\|1\|^2_{\mathscr H_\kappa}=1,&\alpha=0,\\
			\|z_1^{\alpha_1}\|^2_{\mathscr H_\kappa}=\frac{1}{2(\alpha_1+1)},& \alpha_1 \Ge 0, \alpha_2=0,\\
			\|z_2^{\alpha_2}\|^2_{\mathscr H_\kappa}=\frac{1}{2(\alpha_2+1)},& \alpha_1=0, \alpha_2 \Ge 0,\\
			0,& \alpha_1 \Ge 1, \alpha_2 \Ge 1,
		\end{cases}
		\eeqn
and hence we have
\beqn
\|\mathscr T_{z}^\alpha(1)\|^2=\int_{\{0\}\times[0,1]\cup [0,1]\times\{0\}} t^\alpha \, \D(\delta_0 \times \ell + \ell \times \delta_0)/2, \quad \alpha \in \mathbb Z^2_+.
\eeqn
It is now not difficult to see using \eqref{action-Tz-12} that for $f = \sum_{m = 0}^\infty a_m z^m_1 + \sum_{n = 1}^\infty b_n z^n_2$ in $\mathscr H_\kappa/[z_1z_2],$ the representing measure $\mu_f$ of $\{\|\mathscr T_{z}^\alpha(f)\|^2\}_{\alpha \Ge 0}$ is given by  
		\beqn
		\mu_f= |a_0|^2(\ell \times \delta_0 + \delta_0 \times \ell)/2 +\sum_{m = 1}^{\infty}|a_m|^2 (\ell \times\delta_0)/2 + \sum_{n = 1}^{\infty}|b_n|^2 (\delta_0\times \ell)/2.
		\eeqn
Thus 
$\{\|\mathscr T_{z}^\alpha(f)\|^2\}_{\alpha \Ge 0}$ is a Stieltjes moment sequence for every $f \in \mathscr H_\kappa / [z_1z_2].$
It now follows from \cite[Theorem 4.4]{A1987} that $\mathscr H_\kappa/[z_1z_2]$ is subnormal.
\hfill $\diamondsuit$
	\end{example}

We show below that the assumption that $\mathscr H_\kappa$ is subnormal cannot be dropped in Proposition~\ref{sub-linear-u-invariant}. 
Recall that the {\it Dirichlet module} $D_d(\mathbb B^d)$ is the reproducing kernel Hilbert space with kernel $$\kappa(z,w)= \sum_{\alpha \in \mathbb Z^d_+}\frac{|\alpha|!}{(|\alpha|+1)\alpha!}z^\alpha \overline{w}^\alpha, \quad z, w \in \mathbb B^d.$$ 

\begin{example} \label{Diri-module}
Consider the polynomial $p(z_1, z_2)=z_1-az_2,$ where $a$ is a nonzero complex number.  The quotient module $D_2(\mathbb B^2)/[p]$ is not subnormal. 
To see this, recall that the norm on $D_d(\mathbb B^d)$ is given by 
\beq
\label{monomial-Dirichlet}
\|z^\alpha\|^2 = \frac{(|\alpha|+1)\alpha!}{|\alpha|!}, \quad \alpha \in \mathbb Z^d_+. 
\eeq
It follows that
\beqn
\sum_{k=0}^{n} \frac{{|a|}^{2k}}{\|z^k_1 z^{n-k}_2\|^2}= \frac{(1+|a|^2)^n}{n+1}, \quad n \Ge 0.
\eeqn
Note that $\{\gamma_n:=\frac{n+1}{(1+|a|^2)^n}\}_{n \Ge 0}$ is a not Stieltjes moment sequence. Indeed, $\gamma^2_n > \gamma_{2n}$ for every integer $n \Ge 1$. 
Hence, by Lemma~\ref{lem-sub-linear}(iii), the quotient module $D_2(\mathbb B^2)/[p]$ is not subnormal.
\hfill $\diamondsuit$
\end{example}

We now present an example showing that, if $\mathscr H_\kappa$ is a subnormal $\mathbb T^2$-invariant Hilbert module, then the quotient module $\mathscr H_\kappa/[p]$ may fail to be subnormal, even when $\deg\,p=1$.
\begin{example} \label{exam-m-tensor}
For a real number $s > 0,$ consider the Hilbert space $L^2_a(\mathbb D, w_s)$ of holomorphic functions defined on the open unit disc $\mathbb D$, which are square-integrable with respect to the weighted area measure $w_s\,dA$
 with radial weight function $$w_s(z)=\frac{1}{s \pi}|z|^{\frac{2(1-s)}{s}}, \quad z \in \mathbb D.$$
Then $L^2_a(\mathbb D, w_s)$ is a subnormal $\mathbb T$-invariant Hilbert module. 
Also, if $\kappa_s$ is given by 
\beqn
\kappa_s(z, w)=\frac{s}{(1-z\overline{w})^2} + \frac{1-s}{1-z\overline{w}}, \quad z, w \in \mathbb D,
\eeqn
then $L^2_a(\mathbb D, w_s)$ is equal to $\mathscr H_{\kappa_s},$ and satisfies 
\beq
\label{norm-wt-Berg}
\|z^k\|^2_{\mathscr H_{\kappa_{s}}}=\frac{1}{sk+1}, \quad k \in \mathbb Z_+.
\eeq 
For real numbers $s_1, s_2 > 0,$ consider the $\mathbb T^2$-invariant Hilbert module $\widehat{\mathscr H_{\kappa_{s_1}}\otimes\mathscr H_{\kappa_{s_1}}}$ on the bidisc $\mathbb D^2.$
Note that
\beqn 
\inp{z^{\alpha_1}_1 z^{\alpha_2}_2}{z^{\beta_1}_1 z^{\beta_2}_2} &=& \inp{z^{\alpha_1}}{z^{\beta_1}}_{\mathscr H_{\kappa_{s_1}}} \inp{z^{\alpha_2}}{z^{\beta_2}}_{\mathscr H_{\kappa_{s_2}}} \\
&\overset{\eqref{norm-wt-Berg}}=& \begin{cases}  
\frac{1}{(s_1\alpha_1 +1)(s_2\alpha_2 +1)}, & \mbox{if}~ \alpha=\beta \in \mathbb Z^2_+, \\
0, & \mbox{if}~\alpha \neq \beta \in \mathbb Z^2_+,
\end{cases}
\eeqn
and hence by \cite[Lemma 2.1]{BD2004}, $\widehat{\mathscr H_{\kappa_{s_1}}\otimes\mathscr H_{\kappa_{s_1}}}$ is a subnormal Hilbert module. 
Let $p(z_1, z_2)=z_1-z_2.$ 
We may now conclude the following facts from \cite[Corollary 1.6]{AC2017} and Lemma~\ref{lem-sub-linear}(iv):
\begin{enumerate}
\item[$\bullet$] if $({3\mathfrak s - \mathfrak p})^2 \Ge 24\, {\mathfrak p}$ then $\widehat{\mathscr H_{\kappa_{s_1}}\otimes\mathscr H_{\kappa_{s_1}}}/[p]$ is subnormal if and only if $3\mathfrak s > \mathfrak p$,
\item[$\bullet$] if $({3\mathfrak s - \mathfrak p})^2 < 24\, {\mathfrak p}$ then $\widehat{\mathscr H_{\kappa_{s_1}}\otimes\mathscr H_{\kappa_{s_1}}}/[p]$ is subnormal if and only if 
$\mathfrak s \Ge \mathfrak p,$
\end{enumerate}
where $\mathfrak s$ and $\mathfrak p$ denote the sum and product of $s_1, s_2,$ respectively.
In particular, if $\mathfrak s \Ge \mathfrak p,$ then $\widehat{\mathscr H_{\kappa_{s_1}}\otimes\mathscr H_{\kappa_{s_1}}}/[p]$ is subnormal. 
Also, if $3 \mathfrak s \Le \mathfrak p,$ then $\widehat{\mathscr H_{\kappa_{s_1}}\otimes\mathscr H_{\kappa_{s_1}}}/[p]$ is never subnormal. \hfill $\diamondsuit$
\end{example}

We conclude the paper with several unresolved problems.
\begin{enumerate}
\item[$\bullet$] For an integer $d \Ge 3,$ it is not known whether, for every homogeneous polynomial $p$ in $\mathbb C[z_1, \ldots, z_d]$ of degree $1$, the quotient module $H^2(\mathbb D^d)/[p]$ is subnormal. The corresponding problem for $H^2(\mathbb B^d)$ or $H^2_d$ when $d \Ge 3$ also remains open. 
\item[$\bullet$]
Another question arises from Lemmata~\ref{poly-calculus}(i) and \ref{poly-calculus-new}(iii). If $p \in \mathbb C[z_1, z_2]$ is a  homogeneous polynomial and $\mathscr H_\kappa$ is a $\mathbb T^2$-invariant Hilbert module with submodule $[p],$ then $p$ is the polynomial of smallest degree satisfying $p(\mathscr T_{z, [p]})=0.$ We do not know whether this fact extends to dimensions $d \Ge 3.$ 
\item[$\bullet$]
Moreover, it remains unknown whether there exists a nonconstant homogeneous polynomial $p \in \mathbb C[z_1, z_2]$ such that the quotient module $D_2(\mathbb B^2)/[p]$ is subnormal. Since $\mathscr M_z$ on $D_2(\mathbb B^2)$ is a $3$-isometry (see \cite[Corollary~4.9]{CK2014}), it follows from Proposition~\ref{spherical-m-iso-new} and Example~\ref{Diri-module} that
any such polynomial $p$ must have degree either $2$ or $3$. 
However,  $\sigma(\mathscr M_z) \subseteq \overline{\mathbb B^2}$, and $$\|1\|^2-2\|z_j\|^2+\|z^2_j\|^2\overset{\eqref{monomial-Dirichlet}}=0, \quad j=1, 2,$$ and hence Proposition~\ref{spherical-m-iso-new} rules out the possibility that $\deg\,p=3$. Furthermore, explicit computations using Lemma~\ref{coro-generators} suggest that $D_2(\mathbb B^2)/[p]$ is never subnormal unless $p$ is a nonzero constant polynomial.
\end{enumerate}

\section*{Appendix: Square-free polynomials and Hilbert submodules}

In this appendix, we present two basic facts concerning the homogeneous polynomials used in the main body of the paper.
We begin with the existence and uniqueness of the canonical decomposition of a homogeneous polynomial in two variables (cf. \cite[p.~378]{CLO2015}, \cite[Proof of Theorem~2.1.5]{CG2003}).

\begin{propositionA*}
{\it Let $m, n$ be positive integers such that $m \Le n$.
Let $p$ be a homogeneous polynomial given by $p(z_1,z_2)= \sum_{k=0}^{m}\alpha_k z_1^kz_2^{n-k},$ where $\alpha_0, \ldots, \alpha_{m-1} \in \mathbb C$ and $\alpha_m  \in \mathbb C \backslash \{0\}.$ 
Then there exist $a_1, \ldots, a_{m} \in \mathbb{C}$ such that 
\beq \label{cano-de}
p(z)= \alpha_m z_2^{n-m}\prod_{k=1}^{m} (z_1-a_kz_2).
\eeq
Further, if for some $c, d \in \mathbb C\backslash \{0\},$ $r, m \Ge 0,$ $s, n \Ge 1,$ and $a_1,\ldots, a_s, b_1,\ldots, b_n \in \mathbb C,$
\beqn 
c z_2^{r}\prod_{k=1}^{s} (z_1-a_kz_2)=d z_2^{m}\prod_{k=1}^{n} (z_1-b_kz_2), 
\eeqn
then 
$c=d,$ $r=m,$ $n=s,$ and $\{a_1,\ldots,a_s\}=\{b_1,\ldots, b_n\}.$} 
\end{propositionA*}
\begin{proof} 
Note that $p(z_1,z_2)= \alpha_{m} z_2^{n} q(z_1/z_2),$ where $q(w)=\sum_{k=0}^{m}\beta_k w^k$ with $\beta_k = \frac{\alpha_k}{\alpha_m}.$ 
Since $q$ is a polynomial in the complex variable $w$ of degree $m,$ there exist $a_1, \ldots, a_{m} \in \mathbb{C}$ such that $q(w)= \prod_{k=1}^{m} (w-a_k).$ Substituting $w=z_1/z_2,$ we get \eqref{cano-de}.
Since $q$ uniquely determines its zeros (up to a permutation), 
the uniqueness part follows.
\end{proof}
\begin{remarkA*} 
The decomposition in Proposition\,A fails in three complex variables. Indeed, a straightforward computation shows that the homogeneous polynomial $p(z_1, z_2, z_3)=z^2_1+z^2_2+z^2_3$ of degree $2$ is irreducible.
\eof
\end{remarkA*}
 
The following result characterizes square-free homogeneous polynomials in two complex variables. 
\begin{propositionB*} 
{\it A homogeneous polynomial in two complex variables is square-free if and only if it has distinct linear factors.}
\end{propositionB*}
\begin{proof}
The necessity part follows from Proposition\,A.    
To see the sufficiency part, let $p$ be a homogeneous polynomial with distinct linear factors. Suppose $p=q^2r$ for some $q, r \in \mathbb C[z_1, z_2].$ By decomposing a polynomial into homogeneous parts, it is easy to see that any polynomial factor of a homogeneous polynomial must be homogeneous, and hence $r$ and $q$ are homogeneous. 
If $q$ is nonconstant, then by Proposition\,A, $q$ can be factored into linear polynomials, and consequently, $p$ contains squares of linear factors, which is a contradiction. Therefore, $q$ is constant, and hence $p$ is square-free.
\end{proof}

In the remaining part of this appendix, we discuss some applications of Propositions\,A and B to Hilbert modules. First, we show, with the help of an example, that $p\mathscr H_\kappa$ need not be closed in $\mathscr H_\kappa.$
\begin{exampleA*} 
Let $p(z_1, z_2)=z_1-z_2.$ Since 
the sequence $\{\sum_{n=1}^{k} \frac{1}{n}(z^n_1 - z^n_2)\}_{k \Ge 1}$ in $[p]$ is convergent in $H^2(\mathbb D^2)$, 
$f(z_1,z_2)= \sum_{n=1}^{\infty} \frac{1}{n}(z^n_1 - z^n_2)$ belongs to $[p].$ On the other hand, if $f=pg$ for some
$g \in H^2(\mathbb D^2),$ then for $z_1 \neq z_2,$
$$g(z_1, z_2)=\frac{f(z_1,z_2)}{z_1-z_2}= \sum_{n=1}^{\infty} \frac{1}{n} \sum_{k=0}^{n-1} z^k_1 z^{n-1-k}_2,$$ which defines a holomorphic function on $\mathbb D^2$ but does not belong to $H^2(\mathbb D^2)$. Thus $f \notin p H^2(\mathbb D^2).$ In particular, $pH^2(\mathbb D^2) \subsetneq [p].$ 
\hfill $\diamondsuit$
\end{exampleA*}

The following theorem was referred to in Section~\ref{S1}.
\begin{theoremA*}
{\it Let $p$ be a nonzero homogeneous square-free polynomial, and let $\mathscr H_\kappa$ be a $\mathbb T^2$-invariant Hilbert module. Then $[p]_0 = [p].$} 
\end{theoremA*}

Our proof of Theorem\,A depends on the following general fact:
 
\begin{lemmaA*}
{\it Let $\mathscr H_\kappa$ be a $\mathbb T^2$-invariant Hilbert module on $\Omega$.
Let $p$ be a homogeneous polynomial in $\mathbb C[z_1, z_2]$ such that $ph \in \mathscr H_\kappa$ for some holomorphic function $h$ on $\Omega$. Then $ph\in [p]$.}
\end{lemmaA*}
\begin{proof}
Let $p(z)=\sum_{l=0}^{d} \alpha_l z_1^{l} z_2^{d-l}$ be a nonzero homogeneous polynomial of degree $d$ and $ph= f \in  \mathscr H_\kappa.$ For $m \Ge 0,$ let $h_m =  \sum_{n=0}^m \sum_{k=0}^{n}\hat{h}_{k,n-k} z_1^kz_2^{n-k}$ be the $m$-th partial sum of $h$. Since $\mathscr H_\kappa$ contains all polynomials, 
$ph_m \in \mathscr H_\kappa$ for every $m \Ge 0.$ By \eqref{onbasis}, we can write $f$ as $\sum_{n=0}^\infty \sum_{k=0}^{n} \hat{f}_{k, n-k} z_1^kz_2^{n-k}.$
Since $ph= f,$  
\beqn
 \sum_{n=0}^\infty \sum_{l=0}^{d}\sum_{k=0}^{n} \alpha_l \hat{h}_{k,n-k} z_1^{k+l}z_2^{n+d-k-l} 
&=& \sum_{n=0}^\infty \sum_{l=0}^{d} \alpha_l z_1^{l} z_2^{d-l}\sum_{k=0}^{n}\hat{h}_{k,n-k} z_1^kz_2^{n-k}\\
&=&\sum_{n=0}^\infty \sum_{k=0}^{n} \hat{f}_{k, n-k} z_1^kz_2^{n-k},
\eeqn
where rearrangement of terms in these series is possible because of their uniform convergence on compact sets (see \eqref{kappa-cts} and \eqref{rkp}).
Since both the sides are homogeneous decompositions, we obtain
\beqn 
\sum_{l=0}^{d}\sum_{k=0}^{n} \alpha_l \hat{h}_{k,n-k} z_1^{k+l}z_2^{n+d-k-l}=\sum_{k=0}^{n+d} \hat{f}_{k, n+d-k} z_1^kz_2^{n+d-k}, \quad n \Ge 0.
\eeqn
It follows that for integers $m \Ge 0,$
\beqn
    \|ph_m-ph\|^2 &=& \Big\| \sum_{n=m+1}^\infty \sum_{l=0}^{d}\sum_{k=0}^{n} \alpha_l \hat{h}_{k,n-k} z_1^{k+l}z_2^{n+d-k-l}\Big\|^2\\
    &=& \Big\| \sum_{n=m+1}^\infty \sum_{k=0}^{n+d} \hat{f}_{k,n+d-k} z_1^kz_2^{n+d-k}\Big\|^2,
\eeqn
which converges to $0$ as $m \rar \infty$. Hence, $ph_m \in [p]$ converges to $ph$ in ${\mathscr H_\kappa},$ and therefore $ph \in [p].$
\end{proof}

\begin{proof}[Proof of Theorem\,A]
Clearly, $[p] \subseteq [p]_0.$ To see the reverse inclusion, note that by Propositions\,A and B, 
$p$ can be decomposed into distinct linear factors $a_jz_1-b_jz_2$ for some $a_j, b_j \in \mathbb C,$ $1 \Le j \Le d$. Let $f \in \mathscr H_\kappa$ be such that $f=0$ on $Z(p)$. Since $Z(a_1z_1-b_1z_2) \subseteq Z(f)$, by \cite[Lemma~2.1]{BCJ2025}, $\frac{f}{a_1z_1-b_1z_2}$ extends holomorphically to $\Omega$. Thus there exists a holomorphic function $h_1$ on $\Omega$ such that $$f(z_1, z_2) = (a_1z_1-bz_2) h_1(z_1, z_2), \quad (z_1, z_2) \in \Omega.$$ Since $Z(a_2z_1-b_2z_2) \subseteq Z(f)$ and $Z(a_1z_1-b_1z_2) \cap Z(a_2z_1-b_2z_2)=\{(0, 0)\},$ we have $Z(a_2z_1-b_2z_2) \subseteq Z(h_1)$. Once again by \cite[Lemma~2.1]{BCJ2025}, $\frac{h_1}{a_2z_1-b_2z_2}$ extends holomorphically to $\Omega$. Proceeding in the same way, we obtain $f = ph$ for some holomorphic function $h$ on $\Omega$. Hence, by Lemma\,A, $f \in [p].$
\end{proof}

We conclude the appendix with a consequence of Theorem\,A.
\begin{corollaryA*}
{\it Let $p$ be a non-constant homogeneous square-free polynomial, and let $\mathscr H_\kappa$ be a $\mathbb T^2$-invariant Hilbert module on $\Omega$. Then 
\beqn \mathscr H_\kappa /[p]=\bigvee \{\kappa_w : w \in Z(p)\cap \Omega\}.
\eeqn
In particular, $Z(p)\cap \overline{\Omega} \subseteq \sigma(\mathscr T_{z, [p]}) \subseteq Z(p).$}
\end{corollaryA*}
\begin{proof}
Let $\kappa_w :=\kappa(\cdot, w)$ for $w \in \Omega.$ By Theorem\,A, $$[p] =[p]_0= \{f \in \mathscr H_\kappa : \inp{f}{\kappa_w}=0,~w \in Z(p)\cap \Omega\},$$ which is same as $$\big\{f \in \mathscr H_\kappa : f \in \big(\vee \{\kappa_w : w \in Z(p)\cap \Omega\}\big)^{\perp}\big\} = \big(\vee \{\kappa_w : w \in Z(p)\cap \Omega\}\big)^{\perp}.$$
This gives the first part. This, combined with \eqref{adjoint} and the reproducing property \eqref{rkp}, yields the inclusion $Z(p)\cap {\Omega} \subseteq \sigma(\mathscr T_{z, [p]}).$ Since the Taylor spectrum is closed (see \cite[Theorem~3.1]{T1970}), we have $Z(p)\cap \overline{\Omega} \subseteq \sigma(\mathscr T_{z, [p]}).$ The remaining inclusion follows from Lemma~\ref{poly-calculus}(i) and the spectral mapping property (see \cite[Corollary~3.5]{C1988}). 
\end{proof}


\begin{thebibliography}{99}


\bibitem{AC2017} A. Anand, S. Chavan, Module tensor product of subnormal modules need not be subnormal, 
{\it J. Funct. Anal.} {\bf 272} (2017), 4752–4761.

\bibitem{A1998} W. Arveson, Subalgebras of $C^*$-algebras. III. Multivariable operator theory,
{\it Acta Math.} {\bf 181} (1998), 159–228.

\bibitem{A1987} A. Athavale, 
Holomorphic kernels and commuting operators,
{\it Trans. Amer. Math. Soc.} {\bf 304} (1987), 101–110.

\bibitem{A1990} A. Athavale, 
On the intertwining of joint isometries,
{\it J. Operator Theory} {\bf 23} (1990), 339–350. 

\bibitem{BM2001} B. Bagchi and G. Misra, Scalar perturbations of the Nagy-Foias characteristic function, IN Operator Theory : Advances and Application, special volume dedicated to the memory of Bela Sz.-Nagy, {\bf 127} (2001), 97–112.
 
\bibitem{BCJ2025}S. Bera, S. Chavan and S. Jain, A transference principle for involution-invariant functional Hilbert spaces,  {\it J. Geom. Anal.} {\bf 35} (2025), Paper No. 354, pp.~17.

\bibitem{BD2004} C. Berg, A. J. Durán, 
A transformation from Hausdorff to Stieltjes moment sequences, 
{\it Ark. Mat.} {\bf 42} (2004), 239–257.

\bibitem{C2015} S. Chavan, Irreducible tuples without the boundary property, 
{\it Canad. Math. Bull.} {\bf 58} (2015), 9–18.

\bibitem{CK2014} S. Chavan, S. Kumar, Spherically balanced Hilbert spaces of formal power series in several variables. I, {\it J. Operator Theory} {\bf 72} (2014), 405–428.


\bibitem{CM2020} S. Chavan, G. Misra, {\it
Notes on the Brown-Douglas-Fillmore theorem}. Camb. IISc Ser.
Cambridge University Press, Cambridge, 2021. xi+246~pp.

\bibitem{CY2015} S. Chavan, D. Yakubovich, Spherical tuples of Hilbert space operators, {\it Indiana Univ. Math. J.} {\bf 64} (2015), 577–612.

\bibitem{CG2003} X. Chen, K. Guo, {\it Analytic Hilbert modules},
Chapman $\&$ Hall/CRC Research Notes in Mathematics, 433. Chapman $\&$ Hall/CRC, Boca Raton, FL, 2003. viii+201 pp.

\bibitem{C-Z} M. Ch$\bar{\mbox{o}}$ and W. $\dot{\mbox{Z}}$elazko,
On the geometric radius of commuting $n$-tuples of operators, {\it Hokkaido Math. J.} {\bf 21} (1992), 251–258.

\bibitem{Co1991} J. B. Conway, {\it The theory of subnormal operators}, Amer. Math. Soc., Providence, RI,
1991.

\bibitem{CLO2015} D. A. Cox, J. Little, D. O'Shea, Donal, 
{\it Ideals, varieties, and algorithms—An introduction to computational algebraic geometry and commutative algebra}, Fourth edition. Undergraduate Texts in Mathematics. Springer, Cham, 2015. xvi+646~pp.

\bibitem{C1988} R. E. Curto, {\it Applications of several complex variables to multiparameter spectral theory}. In: J. B. Conway and B. B. Morrel (eds.), Surveys of some recent results in operator theory, II, Pitman Research Notes in Mathematics Series, 192, Longman Scientific $\&$ Technical, Harlow, 1988, pp. 25–90.

\bibitem{CS1985} R. E. Curto, N. Salinas, 
Spectral properties of cyclic subnormal $m$-tuples,
{\it Amer. J. Math.} {\bf 107} (1985), 113–138.

\bibitem{DGS2020} B. K. Das, S. Gorai, J. Sarkar, 
On quotient modules of $H^2(\mathbb D^n)$: essential normality and boundary representations, 
{\it Proc. Roy. Soc. Edinburgh Sect. A} {\bf 150} (2020), 1339–1359. 


\bibitem{D1984} R. Douglas, {\it Hilbert modules over function algebras}. Advances in invariant subspaces and other results of operator theory (Timişoara and Herculane, 1984), 125–139,
Oper. Theory Adv. Appl., {\bf 17}, Birkhäuser, Basel, 1986.

\bibitem{D1986} R. Douglas, {\it On Šilov resolution of Hilbert modules}. Special classes of linear operators and other topics (Bucharest, 1986), 51–60,
Oper. Theory Adv. Appl., {\bf 28}, Birkhäuser, Basel, 1988.



\bibitem{DM1993}  R. G. Douglas and G. Misra, Some calculations for Hilbert modules, {\it Journal of the
Orissa Mathematical Society}, {\bf 12} (1993), 75–85.

\bibitem{DM2008} R. G. Douglas, G. Misra, Equivalence of quotient Hilbert modules. II,
{\it Trans. Amer. Math. Soc.} {\bf 360} (2008), 2229–2264.

\bibitem{DMV2000} R. G. Douglas, G. Misra, C. Varughese, 
On quotient modules—the case of arbitrary multiplicity, {\it J. Funct. Anal.} {\bf 174} (2000), 364–398.

\bibitem{DMV2001} R. G. Douglas, G. Misra, C. Varughese, 
{\it Some geometric invariants from resolutions of Hilbert modules}. Systems, approximation, singular integral operators, and related topics (Bordeaux, 2000), 241–270, Oper. Theory Adv. Appl., {\bf 129}, Birkhäuser, Basel, 2001.

\bibitem{DP1989}  R. G. Douglas, V. I. Paulsen, {\it Hilbert modules over function algebras}. Pitman Research Notes in Mathematics Series, 217. Longman Scientific $\&$ Technical, Harlow; copublished in the United States with John Wiley $\&$ Sons, Inc., New York, 1989. vi+130~pp.

\bibitem{DY1990} R. G. Douglas, K. R. Yan, On the rigidity of Hardy submodules,
{\it Integral Equations Operator Theory} {\bf 13} (1990), 350–363.

\bibitem{FR2006} S. H. Ferguson, R. Rochberg, 
{\it Description of certain quotient Hilbert modules}. Operator theory {\bf 20}, 93–109, Theta Ser. Adv. Math., 6, Theta, Bucharest, 2006.

\bibitem{F1994} E. Franks, 
Polynomially subnormal operator tuples,
{\it J. Operator Theory} {\bf 31} (1994), 219–228.

\bibitem{GKMP2024} S. Ghara, S. Kumar, G. Misra, P. Pramanick, 
Commuting tuple of multiplication operators homogeneous under the unitary group, {\it J. Lond. Math. Soc.} {\bf 109} (2024), Paper No. e12890, 37~pp.

\bibitem{GR2006} J. Gleason, S. Richter, $m$-Isometric commuting tuples of operators on a Hilbert space, {\it Integral Equ. Oper. Theory} {\bf 56} (2006) 181–196.

\bibitem{GHX2004} K. Guo, J. Hu, X. Xu,
Toeplitz algebras, subnormal tuples and rigidity on reproducing C$[z_1, \ldots, z_d]$-modules, 
{\it J. Funct. Anal.} {\bf 210} (2004), 214–247.

\bibitem{GW2007} K. Guo and P. Wang,
Essentially normal Hilbert modules and K-homology, III. Homogeneous quotient modules of Hardy modules on the bidisk, {\it Sci. China Ser. A}, {\bf 50} (2007), 387–411.



\bibitem{H2023} M. Hartz, {\it An invitation to the Drury-Arveson space}. Lectures on analytic function spaces and their applications, 347-413. Fields Inst. Monogr., {\bf 39}


\bibitem{JP2008} M. Jarnicki, P. Pflug, {\it
First steps in several complex variables: Reinhardt domains},
EMS Textbooks in Mathematics. European Mathematical Society (EMS), Zürich, 2008. viii+359~pp.

\bibitem{L1977} A. Lubin, 
Weighted shifts and products of subnormal operators,
{\it Indiana Univ. Math. J.} {\bf 26} (1977), 839–845.

\bibitem{M-S} V. M$\ddot{\mbox{u}}$ller and A. Soltysiak, Spectral radius formula for commuting Hilbert space operators,
{\it Studia Math.} {\bf 103} (1992), 329–333.

\bibitem{PR2016} V. I. Paulsen, M. Raghupathi, {\it An introduction to the theory of reproducing kernel Hilbert spaces},
Cambridge Stud. Adv. Math., 152
Cambridge University Press, Cambridge, 2016. x+182~pp.

\bibitem{P1984} M. Putinar, 
Spectral inclusion for subnormal $n$-tuples,
{\it Proc. Amer. Math. Soc.} {\bf 90} (1984), 405–406.

\bibitem{R1988} H. L. Royden, {\it Real analysis},
Third edition. Macmillan Publishing Company, New York, 1988. xx+444~pp.

\bibitem{S1988} N. Salinas, {\it Products of kernel functions and module tensor products}, Topics in operator theory, 219-241, Oper. Theory Adv. Appl, {\bf 32}, Birkh$\ddot{\mbox{a}}$user, Basel, 1988.

\bibitem{Sa2015} J. Sarkar, {\it An introduction to Hilbert module approach to multivariable operator theory}, In: D. Alpay (eds.), Operator Theory. Springer, Basel., (2015), 969–1033. 

\bibitem{Sb2015} J. Sarkar, {\it Applications of Hilbert module approach to multivariable operator theory}, In: D. Alpay (eds.), Operator Theory. Springer, Basel., (2015), 1035–1091. 

\bibitem{SS1989} J. Stochel, F. H. Szafraniec, On normal extensions of unbounded operators. II,
{\it Acta Sci. Math. (Szeged)} {\bf 53} (1989), 153–177.

\bibitem{T1970} J. L. Taylor, A joint spectrum for several commuting operators, {\it J. Funct. Anal.} {\bf 6} (1970), 172–191.


\end{thebibliography}
\end{document}